\documentclass[12pt,reqno]
{amsart}
\usepackage{amsmath,epsfig,graphicx,color}

\usepackage{dsfont}
\usepackage{ifthen}
\usepackage{amssymb,latexsym}
\usepackage{subfigure}
\usepackage{hyperref}
\hypersetup{
urlcolor=black, 
  menucolor=black, 
  citecolor=black, 
  anchorcolor=black, 
  filecolor=black, 
  linkcolor=black, 
  colorlinks=true,
}
\usepackage{comment}
\usepackage[numbers,sort]{natbib} 
\newcommand{\cip}{\stackrel{\P}{\rightarrow}}

\newcommand{\eid}{\stackrel{d}{=}}
\newcommand{\dint}{\,\mathrm{d}}
\newcommand{\x}{\mathbf{x}}
\newcommand{\y}{\mathbf{y}}

\newcommand{\iv}{\mathbf{i}}

\newcommand{\kv}{\mathbf{k}}
\newcommand{\lv}{\ell}
\newcommand{\limn}{\lim_{n \to \infty}}
\usepackage{mathtools}
\mathtoolsset{showonlyrefs}

\definecolor{darkblue}{rgb}{.2, 0.2,.8}
\definecolor{darkgreen}{rgb}{0,0.5,0.3}
\definecolor{darkred}{rgb}{.8, .1,.1}

\newcommand{\ex}{{\rm e}\,}

\newtheorem{lemma}{Lemma}[section]

\newtheorem{theorem}[lemma]{Theorem}

\newtheorem{proposition}[lemma]{Proposition}
\newtheorem{definition}[lemma]{Definition}
\newtheorem{corollary}[lemma]{Corollary}
\newtheorem{example}[lemma]{Example}
\newtheorem{exercise}[lemma]{Exercise}
\newtheorem{remark}[lemma]{Remark}
\newtheorem{fig}[lemma]{Figure}
\newtheorem{tab}[lemma]{Table}

\newcommand{\cid}{\stackrel{d}{\rightarrow}}

\newcommand{\bth}{\begin{theorem}}
\newcommand{\ethe}{\end{theorem}}

\newcommand{\bre}{\begin{remark}\em }
\newcommand{\ere}{\end{remark}}

\newcommand{\ble}{\begin{lemma}}
\newcommand{\ele}{\end{lemma}}

\newcommand{\pp}{point process}
\newcommand{\bde}{\begin{definition}}
\newcommand{\ede}{\end{definition}}
\newcommand{\bco}{\begin{corollary}}
\newcommand{\eco}{\end{corollary}}

\newcommand{\bpr}{\begin{proposition}}
\newcommand{\epr}{\end{proposition}}

\newcommand{\bexer}{\begin{exercise}}
\newcommand{\eexer}{\end{exercise}}

\newcommand{\bexam}{\begin{example}}
\newcommand{\eexam}{\end{example}}

\newcommand{\bfi}{\begin{fig}}
\newcommand{\efi}{\end{fig}}

\newcommand{\btab}{\begin{tab}}
\newcommand{\etab}{\end{tab}}

\newcommand{\sign}{{\rm sign}}

\newcommand{\var}{\operatorname{Var}}
\newcommand{\Var}{\operatorname{Var}}

\newcommand{\as}{{\rm a.s.}}

\newcommand{\rhs}{right-hand side}

\newcommand{\beao}{\begin{eqnarray*}}
\newcommand{\eeao}{\end{eqnarray*}\noindent}

\newcommand{\beam}{\begin{eqnarray}}
\newcommand{\eeam}{\end{eqnarray}\noindent}

\newcommand{\beqq}{\begin{equation}}
\newcommand{\eeqq}{\end{equation}\noindent}

\newcommand{\bce}{\begin{center}}
\newcommand{\ece}{\end{center}}

\newcommand{\barr}{\begin{array}}
\newcommand{\earr}{\end{array}}

\newcommand{\std}{\stackrel{d}{\rightarrow}}

\newcommand{\vague}{\stackrel{\lower0.2ex\hbox{$\scriptscriptstyle
                    \it{v} $}}{\rightarrow}}
\newcommand{\weak}{\stackrel{\lower0.2ex\hbox{$\scriptscriptstyle
                    \it{w} $}}{\rightarrow}}
\newcommand{\what}{\stackrel{\lower0.2ex\hbox{$\scriptscriptstyle
                    \it{\hat{w}} $}}{\rightarrow}}

\newcommand{\bdis}{\begin{displaymath}}
\newcommand{\edis}{\end{displaymath}\noindent}

\newcommand{\N}{\mathbb{N}}
\newcommand{\Z}{\mathbb{Z}}
\newcommand{\R}{\mathbb{R}}

\newcommand{\nto}{n\to\infty}

\newcommand{\ov}{\overline}
\newcommand{\wt}{\widetilde}

\newcommand{\vep}{\varepsilon}

\newcommand{\con}{convergence}

\newcommand{\st}{such that}

\newcommand{\E }{{\mathbb E}}
\renewcommand{\P }{{\mathbb P}}

\newcommand{\1}{\mathds{1}}

\allowdisplaybreaks

\DeclareMathOperator{\e}{e}

\newcommand{\otime}[2]{\mathbin{\mathop{\otimes}\limits_{#1}^{#2}}}

\newcommand{\sumt}{\sum_{t=1}^n}

\makeatletter
\let\save@mathaccent\mathaccent
\newcommand*\if@single[3]{%
  \setbox0\hbox{${\mathaccent"0362{#1}}^H$}%
  \setbox2\hbox{${\mathaccent"0362{\kern0pt#1}}^H$}%
  \ifdim\ht0=\ht2 #3\else #2\fi
  }
\newcommand*\rel@kern[1]{\kern#1\dimexpr\macc@kerna}
\newcommand*\widebar[1]{\@ifnextchar^{{\wide@bar{#1}{0}}}{\wide@bar{#1}{1}}}
\newcommand*\wide@bar[2]{\if@single{#1}{\wide@bar@{#1}{#2}{1}}{\wide@bar@{#1}{#2}{2}}}
\newcommand*\wide@bar@[3]{%
  \begingroup
  \def\mathaccent##1##2{%
    \let\mathaccent\save@mathaccent
    \if#32 \let\macc@nucleus\first@char \fi
    \setbox\z@\hbox{$\macc@style{\macc@nucleus}_{}$}%
    \setbox\tw@\hbox{$\macc@style{\macc@nucleus}{}_{}$}%
    \dimen@\wd\tw@
    \advance\dimen@-\wd\z@
    \divide\dimen@ 3
    \@tempdima\wd\tw@
    \advance\@tempdima-\scriptspace
    \divide\@tempdima 10
    \advance\dimen@-\@tempdima
    \ifdim\dimen@>\z@ \dimen@0pt\fi
    \rel@kern{0.6}\kern-\dimen@
    \if#31
      \overline{\rel@kern{-0.6}\kern\dimen@\macc@nucleus\rel@kern{0.4}\kern\dimen@}%
      \advance\dimen@0.4\dimexpr\macc@kerna
      \let\final@kern#2%
      \ifdim\dimen@<\z@ \let\final@kern1\fi
      \if\final@kern1 \kern-\dimen@\fi
    \else
      \overline{\rel@kern{-0.6}\kern\dimen@#1}%
    \fi
  }%
  \macc@depth\@ne
  \let\math@bgroup\@empty \let\math@egroup\macc@set@skewchar
  \mathsurround\z@ \frozen@everymath{\mathgroup\macc@group\relax}%
  \macc@set@skewchar\relax
  \let\mathaccentV\macc@nested@a
  \if#31
    \macc@nested@a\relax111{#1}%
  \else
    \def\gobble@till@marker##1\endmarker{}%
    \futurelet\first@char\gobble@till@marker#1\endmarker
    \ifcat\noexpand\first@char A\else
      \def\first@char{}%
    \fi
    \macc@nested@a\relax111{\first@char}%
  \fi
  \endgroup
}
\makeatother

\renewcommand{\bar}{\widebar}

\begin{document}
\bibliographystyle{acm}
\title[Point process convergence for symmetric functions]{Point process convergence for symmetric functions of high-dimensional random vectors}
\thanks{We thank Christoph Thäle for fruitful discussions. The work of two anonymous reviewers is gratefully acknowledged.}

\author[J. Heiny]{Johannes Heiny}
\address{Department of Mathematics,
Stockholm University,
Albano hus 1,
10691 Stockholm,
Sweden}
\email{johannes.heiny@math.su.se}
\author[C. Kleemann]{Carolin Kleemann}
\address{Fakult\"at f\"ur Mathematik,
Ruhr-Universit\"at Bochum,
Universit\"atsstrasse 150,
D-44801 Bochum,
Germany}
\email{carolin.kleemann@rub.de}

\begin{abstract}
The convergence of a sequence of point processes with dependent points, defined by a symmetric function of iid high-dimensional random vectors, to a Poisson random measure is proved. This also implies the convergence of the joint distribution of a fixed number of upper order statistics. As applications of the result a generalization of maximum convergence to point process convergence is given for simple linear rank statistics, rank-type U-statistics and the entries of sample covariance matrices.
\end{abstract}
\maketitle

\section{Introduction}\label{sec:intro}
In classical extreme value theory the asymptotic distribution of the maximum of random points plays a central role. Maximum type statistics build popular tests on the dependency structure of high-dimensional data. Especially, against sparse alternatives those tests possess good power properties (see \cite{han:chen:liu:2017,drton2020high,zhou2019extreme}). Closely related to the maxima of random points are point processes, which play an important role in stochastic geometry and data analysis. They have applications in statistical ecology, astrostatistics and spatial epidemiology \cite{baddeley:2007}. For a sequence $(Y_i)_i$ of real-valued random variables, we set
\begin{align*}
    \widetilde{M}_p:=\sum_{i=1}^p\varepsilon_{(i/p,Y_i)},
\end{align*}
where $\varepsilon_x$ is the Dirac measure in $x$. Let $K:=(0,1)\times (u,\infty)$ with $u\in\R$. Then, $\widetilde{M}_p(K)$ counts the number of exceedances of the threshold $u$ by the random variables $Y_1,\ldots, Y_p$. If $Y^{(k)}$ denotes the $k$-th upper order statistic of $Y_1,\ldots,Y_p$, it holds that $\{\widetilde{M}_p(K)<k\}=\{Y^{(k)}\leq u\}$, and in particular $\{\widetilde{M}_p(K)=0\}=\{\max_{i=1,\ldots,p} Y_i\leq u\}$. Therefore, the weak convergence of a sequence of point processes gives information about the joint asymptotic distribution of a fixed number of upper order statistics. If the sequence $(Y_i)_i$ consists of independent and identically distributed (iid) random variables, maximum convergence and point process convergence are equivalent, but if the random variables exhibit dependency, this equivalence does not necessarily hold anymore. In this sense, point process convergence is a substantial generalization of the maximum convergence. Additionally, the time components $i/p$ deliver valuable information of the random time points when a record occurs, i.e., the time points when $Y_j>\max_{i=1,\ldots,j-1}Y_i$. \\
Our main motivation comes from statistical inference for high-dimensional data, where the asymptotic distribution of the maximum of dependent random variables has found several applications in recent years (see for example \cite{han:chen:liu:2017,drton2020high,zhou2019extreme,cai:jiang:2011,cai:liu:xia:2013,cai:2017,cai:liu:2011, goesmann:2022}). The objective of this paper is to provide the methodology to extend meaningful results with reference to the convergence of the maximum of dependent random variables, to point process convergence.

To this end, we consider dependent points $T_{\iv}:=g_{n,p}(\x_{i_1},\x_{i_2},\ldots, \x_{i_m})$,
where the index $\iv= (i_1, i_2,\ldots, i_m)\in \{1,\ldots,p\}^m $. The random vectors $\x_1,\ldots,\x_p$ are iid on $\R^n$ and $g_{n,p}:\R^{mn}\to \R$ is a measurable, symmetric function. Important examples include U-statistics, simple linear rank statistics, rank-type U-statistics, the entries of sample covariance matrices or interpoint distances.

Additionally, we assume that the dimension of the points $n$ is growing with the number of points $p$.
Over the last decades the environment and therefore the requirements for statistical methods have changed fundamentally. Due to the huge improvement of computing power and data acquisition technologies one is confronted with large data sets, where the dimension of observations is as large or even larger than the sample size. These high-dimensional data occur naturally in online networks, genomics, financial engineering, wireless communication or image analysis (see \cite{johnstone:titterington:michael:2009, clarke2008properties, donoho:2000}).
Hence, the analysis of high-dimensional data has developed as a meaningful and active research area.

We will show that the corresponding point process of the points $T_\iv$ converges to a Poisson random measure (PRM) with a mean measure that involves the $m$-dimensional Lebesgue measure and an additional measure $\mu$.
If we replace the points $T_\iv$ with iid random variables with the same distribution, the (non-degenerate) limiting distribution of the maximum will necessarily be an extreme value distribution of the form $\exp(-\mu(x))$. Moreover, the convergence of the corresponding point process will be equivalent to the condition 
\begin{align}\label{cond1}
\binom{p}{m} \P(g_{n,p}(\x_1,\x_2,\ldots, \x_m)>x) \to \mu(x),\qquad p\to\infty.
\end{align}
However, since the random points $T_\iv$ are not independent, we additionally need the following assumption on the dependence structure
\begin{align}\label{cond2i}
\P(g_{n,p}(\x_1,\x_2,\ldots,\x_m)>x, g_{n,p}(\x_{m-l+1},\ldots, \x_{2m-l})>x)=o(p^{-(2m-l)}), \qquad p\to\infty,\phantom{platz} 
\end{align}
where $l=1,\ldots, m-1$.

In the finite-dimensional case where $n$ is fixed, several results about point process convergence are available in similar settings. 
In \cite{MR511059}, Silverman and Brown showed point process convergence for $m=2$, $n=2$ and $g_{2,p}(\x_i,\x_j)=a_p\Vert \x_i-\x_j\Vert_2^2$, where the $\x_i$ have a bounded and almost everywhere continuous density, $a_p$ is a suitable scaling sequence and $\Vert\cdot\Vert_2$ is the Euclidean norm on $\R^2$. In the Weibull case  $\mu(x)= x^\alpha$ for $x,\alpha>0$, Dehling et al.~\cite{dabrowski2002poisson} proved a generalization to points with a fixed dimension and $g_{n,p}(\x_i,\x_j)=a_ph(\x_i,\x_j)$, where $h$ is a measurable, symmetric function and $a_p$ is a suitable scaling sequence.  


Also in the finite-dimensional case, under similar assumptions as in \eqref{cond1} with $\mu(x)=\beta x^\alpha$ for $x,\alpha>0$, $\beta\in\R$ and under condition \eqref{cond2i}, Schulte and Thäle \cite{schulte:thaele:2012} showed convergence in distribution of point processes towards a Weibull process. The points of these point processes are obtained by applying a symmetric function $g_{n,p}$ to all $m$-tuples of distinct points of a Poisson process on a standard Borel space. In \cite{schulte2016poisson}, this result was extended to more general functions $\mu$ and to binomial processes so that other PRMs were possible limit processes. 
In \cite{decreusefond:schulte:thaele:2016}, Decreusefond, Schulte and Thäle provided an upper bound of the Kantorovich–Rubinstein distance between a PRM and the point process induced in the aforementioned way  by a Poisson or a binomial process on an abstract state space. Notice that convergence in Kantorovich-Rubinstein distance implies convergence in distribution  (see \cite[Theorem 2.2.1]{panaretos:2020} or \cite[p.~2149]{decreusefond:schulte:thaele:2016}). 
 In \cite{chenavier:2022} another point process result in a similar setting is given for the number of nearest neighbor balls in fixed dimension. Moreover, \cite{basrak:2021} presents a general framework for Poisson approximation of point processes on Polish spaces.

 \subsection{Structure of this paper}
The remainder of this paper is structured as follows. In Section \ref{sec1} we prove weak point process convergence for the dependent points $T_{\iv}$ in the high-dimensional case as tool for the generalization of the convergence of the maximum (Theorem~\ref{thm:maxg}).
 We provide popular representations of the limiting process in terms of the transformed points of a homogeneous Poisson process. Moreover, 
 we derive point process convergence for the record times. In Section \ref{sec2} these tools are applied to study statistics based on relative ranks like simple linear rank statistics or rank-type U-statistics. We also prove convergence of the point processes of the off-diagonal entries of large sample covariance matrices. 
 The technical proofs are deferred to Section~\ref{proof}.

\subsection{Notation}
Convergence in distribution (resp.\ probability) is denoted by $\cid$ (resp.\ $\cip$) and unless explicitly stated otherwise all limits are for $\nto$. For sequences $(a_n)_n$ and $(b_n)_n$ we write $a_n=O(b_n)$ if $a_n/b_n\leq C$ for some constant $C>0$ and every $n\in\N$, and $a_n=o(b_n)$ if $\lim_{n\to\infty} a_n/b_n=0$. Additionally, we use the notation
$a_n\sim b_n$ if $\lim_{n\to\infty} a_n/b_n=1$ and $a_n\lesssim b_n$ if $a_n$ is smaller than or equal to $b_n$ up to a positive universal constant. We further write $a\wedge b:=\min\{a,b\}$ for $a,b\in\R$ and for a set $A$ we denote $|A|$ as the number of elements in $A$.

\section{Point process convergence}\setcounter{equation}{0}\label{sec1}
We introduce the model that was briefly described in the introduction. Let $\x_1,\ldots,\x_p$ be iid $\R^n$-valued random vectors with $\x_i=(X_{i1}, \ldots, X_{in})^\top, i=1,\ldots,p$, where $p=p_n$ is some positive integer sequence tending to infinity as $\nto$.\\
We consider the random points 
\beam
T_{\iv}:=g_{n,p}(\x_{i_1},\x_{i_2},\ldots, \x_{i_m}),
\eeam
where $\iv=(i_1, i_2,\ldots, i_m)\in \{1,\ldots,p\}^m$ and $g_n=g_{n,p}:\R^{mn}\to \R$ is a measurable and symmetric function, where symmetric means $g_{n}(\y_1,\y_2,\ldots, \y_m)=g_{n}(\y_{\pi(1)},\y_{\pi(2)},\ldots, \y_{\pi(m)})$ for all $\y_1,\y_2,\ldots, \y_m \in \R^n$ and all permutations $\pi$ on $\{1,2,\ldots, m\}$.
We are interested in the limit behavior of the \pp es $ M_n$ towards a PRM\ $M$,
\beam\label{eq:conv}
 M_n=\sum_{1\le i_1<i_2<\ldots<i_m\le p} \vep_{(\iv/p,\,T_{\iv})} \std M\,,\qquad \nto\,,
\eeam
where $\iv/p=(i_1/p,\ldots,i_m/p)$. The limit $M$ is a PRM\ with mean measure 
\begin{align*}
\eta\Big(\otime{l=1}{m+1}(r_l,s_l)\Big)=m!{}\lambda_m\Big(\otime{l=1}{m}(r_l,s_l)\Big)\mu(r_{m+1},s_{m+1}),
\end{align*}
where $\lambda_m$ is the Lebesgue measure on $\R^m$. For an interval $(a,b)$ with $a<b\in\R$ we write $\mu(a,b):=\mu((a,b)):=\mu(a)-\mu(b)$ and $\mu:(v,w)\to \R^+=\{x\in \R:x\ge 0\}$ is a function satisfying $\lim_{x \to v} \mu(x)=\infty$ and $\lim_{x \to w} \mu(x)=0$ for $v,w\in\bar{\R}=\R\cup\{\infty, -\infty\}$ and $v<w$.
Furthermore, we set $\eta_n(\cdot):=\E[M_n(\cdot)]$.
We consider the $M_n$'s and $M$ as random measures on the state space
\begin{align*}
S=S_1\times (v,w)=\{(z_1, z_2,\ldots, z_m): 0<z_1<z_2<\ldots<z_m\leq 1\}\times(v,w)
\end{align*}
with values in $\mathcal{M}(S)$ the space of point measures on $S$, endowed with the vague topology (see \cite{resnick:1987}). The following result studies the convergence $M_n \cid M$, which denotes the convergence in distribution in $\mathcal{M}(S)$.

\begin{theorem}\label{thm:maxg}
Let $\x_1,\ldots, \x_p$ be $n$-dimensional, independent and identically distributed random vectors and $p=p_n$ is some sequence of positive integers tending to infinity as $\nto$. Additionally, let $g=g_n:\R^{mn}\to (v,w)$ be a measurable and symmetric function, where $v,w\in\bar{\R}=\R\cup\{\infty, -\infty\}$ and $v<w$. Assume that there exists a function $\mu:(v,w)\to \R^+$ with $\lim_{x \to v} \mu(x)=\infty$ and $\lim_{x \to w} \mu(x)=0$ 
such that, for $x\in (v,w)$ and $\nto$,
\begin{itemize}
\item[(A1)] $\binom{p}{m} \P(g_n(\x_1,\x_2,\ldots, \x_m)>x) \to \mu(x)$ and
\item[(A2)] $\P(g_n\!(\x_1,\x_2,\ldots,\x_m)\!>\!x,g_n\!(\x_{m-l+1},\ldots, \x_{2m-l})\!>\!x)\!=\!o(p^{-(2m-l)})$ for \mbox{$l=\!1,\ldots, m\!-\!1$.} 
\end{itemize}
Then we have $M_n\cid M$.
\end{theorem}

 Note that (A1) ensures the correct specification of the mean measure, while (A2) is an anti-clustering condition. Both conditions are standard in extreme value theory. It is worth mentioning that  
$$\limn \P\Big(\max_{1\le i_1<i_2<\ldots<i_m\le p} T_{\iv}\le x\Big)=\exp(-\mu(x))=:H(x)\,, \qquad x\in \R\,,$$
where we use the conventions $\mu(x)=0$ if $x>w$, $\mu(x)=\infty$ if $x<v$, and $\exp(-\infty)=0$.
The typical distribution functions $H$ are the Fr\'{e}chet, Weibull and Gumbel distributions. 
In these cases, the limiting process $M$ has a representation in terms of the transformed points of a homogeneous Poisson process.
Let $(U_i)_i$ be an iid sequence of random vectors uniformly distributed on $S_1$ and $\Gamma_i=E_1+\ldots+E_i$, where $(E_i)_i$ is an iid sequence of standard exponentially distributed random variables, independent of $(U_i)_i$.\\
It is well--known that $N_\Gamma:=\sum_{i=1}^\infty \varepsilon_{\Gamma_i}$ is a homogeneous Poisson process and hence it holds for every $A\subset(0,\infty)$ that $N_{\Gamma}(A)$ is Poisson distributed with parameter $\lambda_1(A)$ (see for example \cite[Example 5.1.10]{embrechts:kluppelberg:mikosch:1997}). For the mean measure $\eta$ of $M$ we get for a product of intervals $\otime{l=1}{m}(r_l,s_l]\subset S_1$
\begin{align*}
\eta\Big(\otime{l=1}{m+1}(r_l,s_l]\Big)&=
m!\lambda_m\Big(\otime{l=1}{m}(r_l,s_l]\Big)(\mu(r_{m+1})-\mu(s_{m+1}))\\
&=\sum_{i=1}^\infty\P\Big(U_i\in \otime{l=1}{m}(r_l,s_l]\Big)\E[\varepsilon_{\Gamma_i}(\mu(s_{m+1}),\mu(r_{m+1})]\\
&=\E\Big[\sum_{i=1}^\infty\varepsilon_{(U_i,\Gamma_i)}\Big(\otime{l=1}{m}(r_l,s_l]\times(\mu(s_{m+1}),\mu(r_{m+1})]\Big)\Big],
\end{align*}
where we used in the second line that 
\begin{align*}
\P\Big(U_i\in \otime{l=1}{m}(r_l,s_l]\Big)=m!\prod_{l=1}^m(s_l-r_l)
\end{align*}
as $U_i$ is uniformly distributed on $S_1$ for every $i$ and
\begin{align*}
\E[N_\Gamma(\mu(s_{m+1}),\mu(r_{m+1}))]=\lambda_1(\mu(s_{m+1}),\mu(r_{m+1})).    
\end{align*}

We get the following representations for the limiting processes $M$. 
\begin{itemize}
\item Fr\'{e}chet case:
For $\alpha>0$ the Fr\'{e}chet distribution is given by $\Phi_\alpha(x)=\exp({-x^{-\alpha}})$, \mbox{$x>0$}. For $0<r<s<\infty$ we have $\mu_{\Phi_\alpha}(r,s]=r^{-\alpha}-s^{-\alpha}$ and therefore, we can write
\begin{align*}
M=M_{\Phi_\alpha}=\sum_{i=1}^\infty\varepsilon_{(U_i,\Gamma_i^{-1/\alpha})}.
\end{align*}
\item Weibull case:
For $\alpha>0$ the Weibull distribution is given by $\Psi_\alpha(x)=\exp({-|x|^\alpha})$, \mbox{$x<0$}. For $-\infty<r<s<0$ we have $\mu_{\Psi_{\alpha}}(r,s]=|r|^\alpha-|s|^\alpha$ and
\begin{align*}
M=M_{\Psi_\alpha}=\sum_{i=1}^\infty\varepsilon_{(U_i,-\Gamma_i^{1/\alpha})}.
\end{align*}
\item Gumbel case:
The Gumbel distribution is given by $\Lambda(x)=\exp({-\e^{-x}})$ for all $x\in\R$. For $-\infty<r<s<\infty$ we have $\mu_{\Lambda}(r,s]=\e^{-r}-\e^{-s}$ and
\begin{align*}
M=M_\Lambda=\sum_{i=1}^\infty\varepsilon_{(U_i,-\log \Gamma_i)}.
\end{align*}
\end{itemize}

Besides the points $T_\iv$, time components $\iv/p=(i_1/p,\ldots,i_m/p)$ with $1\leq i_1\leq\ldots\leq i_m\leq p$ are considered in the definition of the point process $M_n$.
Whenever we do not need the time components in the following, we will use the shorthand notation
\begin{align}\label{eq:defNnsdfsddv}
N_n(\cdot):=M_n(S_1\times\cdot)=\sum_{1\leq i_1\leq\ldots\leq i_m\leq p}\varepsilon_{T_\iv}(\cdot).
\end{align} 
Under the conditions of Theorem \ref{thm:maxg}, $N_n$ converges in distribution to $N(\cdot):=M(S_1\times\cdot)$ which is a PRM with mean measure $\mu$.

A direct consequence of the point process convergence is the convergence of the joint distribution of a fixed number of upper order statistics. In the Fr\'{e}chet, Weibull and Gumbel cases the limit function can be described as the joint distribution function of transformations of the points $\Gamma_i$.
\begin{corollary}\label{cor:ssdgsd}
Let $G_{n,(j)}$ be the $j$-th upper order statistic of the random variables \\$(g_n(\x_{i_1},\x_{i_2},\ldots,\x_{i_m}))$, where $1\leq i_1<i_2<\ldots<i_m\leq p$. Under the conditions of Theorem~\ref{thm:maxg} and for a fixed $k\ge 1$ the distribution function
\begin{align*}
\P(G_{n,(1)}\leq x_1,\ldots,G_{n,(k)}\leq x_k), 
\end{align*}
where $x_k<\ldots<x_1\in (v,w)$, converges to 
\begin{align*}
\P\Big(N(x_1,w)= 0, N(x_2,w)\leq 1\ldots,N(x_k,w)\leq k-1\Big),
\end{align*}
as $n\to\infty$. In particular, in the Fr\'{e}chet, Weibull and Gumbel cases, it holds that 
\begin{align*}
 &\P\Big(N(x_1,w)= 0, N(x_2,w)\leq 1\ldots,N(x_k,w)\leq k-1\Big)\\
 &= \begin{cases}
  \P(\Gamma_1^{-1/\alpha}\leq x_1,\ldots, \Gamma_k^{-1/\alpha}\leq x_k),\quad &\text{if}\,\,\,\mu=\mu_{\Phi_\alpha}, \\
  \P(-\Gamma_1^{1/\alpha}\leq x_1,\ldots, -\Gamma_k^{1/\alpha}\leq x_k),\quad &\text{if}\,\,\,\mu=\mu_{\Psi_\alpha},\\
  \P(-\log \Gamma_1\leq x_1,\ldots, -\log \Gamma_k\leq x_k),\quad &\text{if}\,\,\,\mu=\mu_{\Lambda}.
 \end{cases}
\end{align*}
\end{corollary}

\begin{proof}
Since $N_n(x,w)$ is the number of vectors $\iv=(i_1,\ldots,i_m)$ with $1\le i_1<i_2<\ldots<i_m\le p$, for which $g_n(\x_{i_1},\x_{i_2},\ldots,\x_{i_m})\in (x,w)$, we get by Theorem \ref{thm:maxg} as $n\to\infty$
{\small
\begin{align}
\P(G_{n,(1)}\leq x_1,\ldots,G_{n,(k)}\leq x_k)&=\P\Big(N_n(x_1,w)= 0,N_n(x_2,w)\leq 1,\ldots,N_n (x_k,w)\leq k-1\Big) \nonumber\\
&\to\P\Big(N(x_1,w)= 0,N(x_2,w)\leq 1\ldots,N(x_k,w)\leq k-1\Big).
\label{df}
\end{align}}
By the representation of the limiting point process in the Fr\'{e}chet, Weibull and Gumbel cases, \eqref{df} is equal to one of the three distribution functions in the corollary.
\end{proof}
 One field, where point processes find many applications, is stochastic geometry. The paper \cite{schulte:thaele:2012}, for example, considers order statistics for Poisson $k$-flats in $\R^d$, Poisson polytopes on the unit sphere and random geometric graphs.


Setting $k=1$ in Corollary~\ref{cor:ssdgsd} we obtain the convergence in distribution of the maximum of the points $T_\iv$.
\begin{corollary}\label{cor:maxg}
Under the conditions of Theorem~\ref{thm:maxg} we get 
\begin{align*}
 \lim_{\nto} \P\Big(\max_{1\le i_1<i_2<\ldots<i_m\le p} g_n(\x_{i_1},\x_{i_2},\ldots,\x_{i_m})\le x\Big) = \exp(-\mu(x)) \,, \qquad x\in \R\,.\\
\end{align*}
\end{corollary}

\begin{figure}[h]
    \centering
    \includegraphics[width=0.5\textwidth]{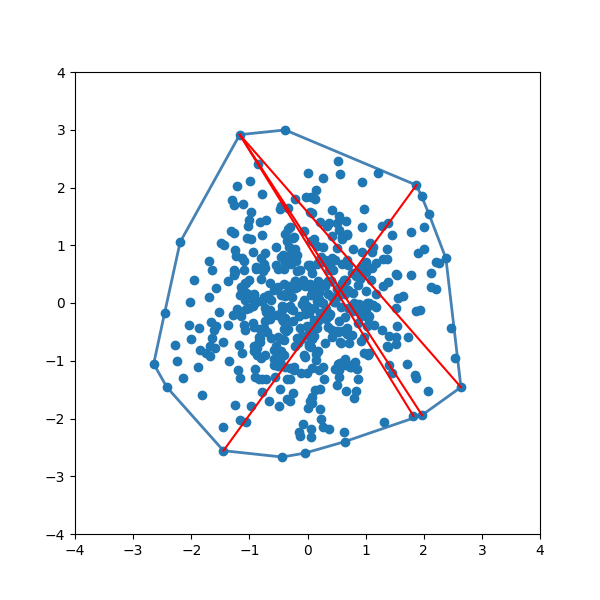}
    \caption{Four largest distances between 500 normal distributed points}
    \label{fig:my_label}
\end{figure}
\begin{example}[Interpoint distances]\em{
Let $\x_i=(X_{i1}, \ldots, X_{in})^\top, i=1,\ldots,p$ be $n$-dimensional random vectors, whose components $(X_{it})_{i,t\ge 1}$ are independent and identically distributed random variables with zero mean and variance $1$.
We are interested in the asymptotic behavior of the largest interpoint distances
\begin{equation}\label{eq:defDij}
D_{ij}= \| \x_i -\x_j \|^2_2= \sum_{t=1}^n(X_{it}-X_{jt})^2\,, \qquad 1\le i <j\le p\,,
\end{equation} 
where $\|\cdot\|_2$ is the Euclidean norm on $\R^n$. Figure~\ref{fig:my_label} shows the four largest interpoint distances of $500$ points on $\R^2$ with independent standard normal distributed components. Note that three of the largest four distances involve the same outlying vector $\x_i$.

We assume that there exists $s>2$ \st\ $\E[|X_{11}|^{2s}(\log(|X_{11}|))^{s/2}]< \infty$ and $\E[X_{11}^4]\leq 5$ and that $p=p_n\to\infty$ satisfies $p=O(n^{(s-2)/4})$. Additionally, we let $(b_n)_n$ and $(c_n)_n$ be sequences given by
\begin{equation*}
b_n=2n+\sqrt{2n(\E[X^4]+1)}d_n\quad \text{ and } \quad c_n=\frac{d_n}{\sqrt{2n(\E[X^4]+1)}}\,,
\end{equation*}
where $d_n=\sqrt{2\log \tilde{p}} - \tfrac{\log\log \tilde{p}+\log 4\pi}{2(2\log \tilde{p})^{1/2}}$
with $\tilde{p}=p(p-1)/2$. For $x\in \R$ one can check that
$$\tilde{p}\, \P\big( c_n (D_{12}-b_n)>x\big) \to \e^{-x} \quad \text{and} \quad
\P\big( c_n (D_{12}-b_n)>x, c_n(D_{23}-b_n)>x\big)=o(p^{-3})$$
as $\nto$ (see \cite{heiny2023maximum} for details). Therefore, the conditions (A1) and (A2) in Theorem~\ref{thm:maxg} hold for $m=2$, $g_n(\x_i,\x_j)=c_n (D_{ij}-b_n)$ and $\mu(x)=\e^{-x}$. By virtue of Theorem \ref{thm:maxg} we have 
\begin{align*}
\sum_{1\leq i<j\leq p}\varepsilon_{c_n (D_{ij}-b_n)}\cid N_\Lambda=\sum_{i=1}^\infty\varepsilon_{-\log \Gamma_i}.
\end{align*} 
Finally Corollary~\ref{cor:ssdgsd} yields for a fixed $k\geq 1$ that
\begin{align*}
(D_{n,(1)},\ldots,D_{n,(k)})\cid (-\log \Gamma_1,\ldots, -\log \Gamma_k), 
\end{align*}
where $D_{n,(\ell)}$ is the $\ell$-th upper order statistic of the random variables $c_n(D_{ij}-b_n)$ for $1\leq i<j\leq p$.}
\end{example}

\subsection*{Record times}
In Theorem \ref{thm:maxg} we showed convergence of point processes including time components. Therefore, we can additionally derive results for the record times $L(k), k\geq 1$ of the running maxima of the points $T_\iv=g_n(\x_{i_1},\x_{i_2},\ldots,\x_{i_m})$
    for $\iv=(i_1,\ldots,i_m)$, which are recursively defined as follows:
\begin{align*}
    L(1)&=1\,,\\
    L(k+1)&=\inf\{j>L(k):\max\limits_{1\le i_1<\ldots<i_m\le j} T_\iv>\max\limits_{1\le i_1<\ldots<i_m\le L(k)} T_\iv\},\qquad k\in\N,
\end{align*}
(c.f.~Sections 5.4.3 and 5.4.4 of \cite{embrechts:kluppelberg:mikosch:1997}). To prove point process convergence for the record times we need the convergence in distribution of the sequence of processes $(Y_n(t), 0<t\leq 1)$ in $D(0,1]$, the space of right continuous functions on
$(0,1]$ with finite limits existing from the left, defined by
\begin{align*}
    Y_n(t)=\begin{cases} \max\limits_{1\le i_1<i_2<\ldots<i_m\le \lfloor pt\rfloor} g_n(\x_{i_1},\x_{i_2},\ldots,\x_{i_m}),\quad & \frac{m}{p}\leq t\leq 1\\
    g_n(\x_{1},\x_{2},\ldots,\x_{m}), & \text{otherwise},
    \end{cases}
\end{align*}
where  $\lfloor x\rfloor=\max\{y\in\Z:y\leq x\}$ for $x\in\R$, towards an extremal process. We call $Y=(Y(t))_{t>0}$ an extremal process generated by the distribution function $H$, if the finite-dimensional distributions are given by 
\begin{align}\label{extremal}
    \P(Y(t_1)\leq x_1,\ldots, Y(t_k)\leq x_k)=H^{t_1}\Big(\bigwedge_{i=1}^k x_i\Big)H^{t_2-t_1}\Big(\bigwedge_{i=2}^k x_i\Big)\ldots H^{t_k-t_{k-1}}( x_k),
\end{align}
where $k\geq 1$, $0<t_1<\ldots<t_k$, $x_i\in \R$ and $1\leq i\leq k$ (see \cite[Definition 5.4.3]{embrechts:kluppelberg:mikosch:1997}).
To define convergence in distribution in $D(0,1]$ we first need to introduce a metric $\mathcal{D}$ on $D(0,1]$. 
To this end, let $\Lambda_{[0,1]}$ be a set of homeomorphisms
\begin{align*}
    \Lambda_{[0,1]}=\{\lambda:[0,1]\to[0,1]: \lambda(0)=0, \lambda(1)=1, \lambda \text{ is continuous and strictly increasing}\}.
\end{align*}
Then for $f,g\in D[0,1]$ the Skorohod metric $\tilde{\mathcal{D}}$ is defined by (see \cite[Section 12]{billingsley:1999})
\begin{align*}
    \tilde{\mathcal{D}}(f,g):=\inf\{&\epsilon>0: \text{there exists a}\,\, \lambda\in\Lambda_{[0,1]}\,\, \text{such that}\\
    &\sup_{0\leq t\leq 1}|\lambda(t)-t|\leq \epsilon,\,\sup_{0\leq t\leq 1}|f(t)-g(\lambda(t))|\leq \epsilon\}.
\end{align*}
Now set
\begin{align*}
    \mathcal{D}(f,g):=\tilde{\mathcal{D}}(\tilde{f},\tilde{g}),\quad f,g\in D(0,1],
\end{align*}
where $\tilde{f}$ and $\tilde{g}$ are the right continuous extensions of $f$ and $g$ on $[0,1]$.The space of functions $D[0,1]$ and therefore $D(0,1]$ is separable under the Skorohod metric but not complete. However, one can find an equivalent metric, i.e., a metric which generates the same Skorohod topology, under which $D[0,1]$ is complete (see \cite[Theorem 12.2]{billingsley:1999}). In particular, the Skorohod metric and the equivalent metric generate the same open sets and thus the $\sigma$-algebras of the Borel sets, which are generated by these open sets, are the same. Therefore, a sequence of probability measures on $D(0,1]$ is relatively compact if and only if it is tight \cite[Section 13]{billingsley:1999}. Hence, for every tight sequence of probability measures on $D(0,1]$ the convergence of the finite dimensional distributions on all continuity points of the limit distribution implies convergence in distribution \cite[Theorem 13.1]{billingsley:1999}.

For the PRM $M=\sum_{i=1}^\infty\varepsilon_{(U_i,\Delta_i)}$, where $(U_i)_i$ is an iid sequence of random vectors uniformly distributed on $S_1$ and
\begin{align*}
    \Delta_i=\begin{cases}
        -\log(\Gamma_i)\quad &\text{if}\,\,\, H=\Lambda,\\
        \Gamma_i^{- 1/\alpha}\quad &\text{if}\,\,\, H=\Phi_\alpha,\\
        -\Gamma_i^{1/\alpha}\quad &\text{if}\,\,\, H=\Psi_\alpha,
    \end{cases}
\end{align*}
we set
\begin{align*}
    Y(t)=\sup\{\Delta_i:U_i^{(m)}\leq t, i\geq 1\}\,,\qquad t\in (0,1]\,,
\end{align*}
where $U_i^{(m)}$ is the $m$-th component of $U_i$. Then the process $Y$ has the finite dimensional distributions in \eqref{extremal}
for $k\geq 1$, $0<t_i\leq 1$, $x_i\in \R$ and $1\leq i\leq k$. Therefore, $Y$ is an extremal process generated by $H$ restricted to the interval $(0,1]$.
For these processes we can show the following invariance principle by application of the continuous mapping theorem (see \cite[Theorem 2.7]{billingsley:1999} or \cite[p. 152]{resnick:1987}).
\begin{proposition}\label{prop:extr}
Under the conditions of Theorem \ref{thm:maxg} and if $H(\cdot)=\exp(-\mu(\cdot))$ is an extreme value distribution it holds that
\begin{align*}
    Y_n\cid Y\,, \qquad \nto\,,
\end{align*}
in $D(0,1]$ with respect to the metric $\mathcal{D}$. 
\end{proposition}
Since $Y$ is a nondecreasing function, which is constant between isolated jumps, it has only countably many discontinuity points. Now let $(\tau_n)_n$ be the sequence of these discontinuity points of $Y$.  Notice that by \cite[Theorem 5.4.7]{embrechts:kluppelberg:mikosch:1997} the point process $\sum_{k=1}^\infty\varepsilon_{\tau_k}$ is a PRM with mean measure $\nu(a,b)=\log(b/a)$ for $0<a<b\leq 1$.
We are ready to state our result for the point process of record times.
\begin{theorem}\label{thm:rectime}
Under the conditions of Theorem \ref{thm:maxg} and if $H(\cdot)=\exp(-\mu(\cdot))$ is an extreme value distribution it holds that
\begin{align*}
    J_n:=\sum_{k=1}^p\varepsilon_{p^{-1}L(k)}\cid J:=\sum_{k=1}^\infty \varepsilon_{\tau_k},
\end{align*}
in $\mathcal{M}(0,1]$, the space of point measures on $(0,1]$. 
\end{theorem}
Based on Theorem \ref{thm:rectime} we can make statements about the time points of the last and second last record at or before $p$.
\begin{corollary}\label{rem:recdis}
 Assume the conditions of Theorem~\ref{thm:rectime} and let $\zeta(p)$ be the number of records among the random variables $$\max_{1\leq i_1,\ldots, i_m\leq m}T_{\iv},\ldots, \max_{1\leq i_1,\ldots, i_m\leq p}T_{\iv}\,.$$ 
 Then the following statements hold for $x,y\in (0,1]$ as $n\to \infty$.
\begin{enumerate}
    \item[(1)] $\P(p^{-1}L(\zeta(p))\leq x)=\P(J_n(x,1]=0)\to \P(J(x,1]=0)=x$.\\ 
    
    \item[(2)] $\P(p^{-1}L(\zeta(p))\leq x, p^{-1}L(\zeta(p)-1)\leq y)\to y+y\log(x/y)$ for $x>y$.\\
    
    \item[(3)] $\P(p^{-1}(L(\zeta(p))-L(\zeta(p)-1))\leq x)\to x(1-\log(x))$.
\end{enumerate}

\end{corollary}

\begin{proof}
Let $0<y<x\le 1$. Part (1) is a direct consequence of the definitions of $\zeta$ and $L$. Part (2) follows by
\begin{align*}
    \P(p^{-1}L(\zeta(p))\leq x, p^{-1}L(\zeta(p)-1)\leq y)&=\P(J_n(x,1]=0, J_n(y,1]\leq 1)\\&\to \P(J(x,1]=0, J(y,1]\leq 1)
\end{align*}
as $n\to\infty$ and
\begin{align*}
    \P(J(x,1]=0, J(y,1]\leq 1)=\P(J(x,1]=0)\P(J(y,x]\leq 1)=y+y\log(x/y).
\end{align*}
To prove part (3) we assume that $\tau^{(1)}$ and $\tau^{(2)}$ are the first and the second upper order statistics of $(\tau_n)_n$. These upper order statistics exist since for every $a>0$ there are only finitely many $\tau_n\in [a,1]$. Then, we know by part (2) that
\begin{align}\label{distrfun}
    \P(\tau^{(1)}\leq x, \tau^{(2)}\leq y)=\P(J(x,1]=0, J(y,1]\leq 1)=\begin{cases} y+y\log(x/y),\quad &x>y\\
    x, &\text{otherwise}.
    \end{cases}
\end{align}
Since
\begin{align*}
   \lim_{\nto} \P(p^{-1}(L(\zeta(p))-L(\zeta(p)-1))\leq x)=  \P(\tau^{(1)}-\tau^{(2)}\leq x)
\end{align*}
we need to calculate $\P(\tau^{(1)}-\tau^{(2)}\leq x)$.
The joint density of $\tau^{(1)}$ and $\tau^{(2)}$ can be deduced from \eqref{distrfun}, it is
\begin{align*}
    f_{\tau^{(1)}\tau^{(2)}}(u,v)=\begin{cases} 1/u\quad &u>v\\
    0, &\text{otherwise}.
    \end{cases}
\end{align*}
Hence, we get the following distribution function of $\tau^{(1)}-\tau^{(2)}$
\begin{align*}
    \P(\tau^{(1)}-\tau^{(2)}\leq x)&=\int_0^x\int_0^{1-w} f_{\tau^{(1)}\tau^{(2)}}(w+v,v)dv\, dw\\
    &=\int_0^x\int_0^{1-w} 1/(w+v) dv \,dw = \int _0^x \log(1/w) dw = x(1-\log(x)),
\end{align*}
which completes the proof.
\end{proof}

\section{Applications}\setcounter{equation}{0}\label{sec2}

\subsection{Relative ranks}

In recent years, maximum-type tests based on the convergence in distribution of the maximum of rank statistics of a data set have gained significant interest for statistical testing \cite{han:chen:liu:2017}. Let $\y_1,\ldots,\y_n$ be $p$-dimensional iid random vectors with $\y_t=(X_{1t},\ldots,X_{pt})$ following a continuous distribution to avoid ties. 
 We write $Q_{it}$ for the rank of $X_{it}$ among $X_{i1},\ldots, X_{in}$. Additionally, let $R_{ij}^{(t)}$ be the relative rank of the $j$-th entry compared to the $i$-th entry; that is  $R_{ij}^{(t)} = Q_{ j t'}$ with $t'$ such that $Q_{i t'}=t$ for $t=1,\ldots,n$. 

A simpler explanation of $R_{ij}^{(t)}$ is that we look at the $j$-th and $i$-th rows of $(Q_{it})$ and find the location of $t$ in the $i$-th row. Then we choose the value in the $j$-th row at this location.

Many important statistics are based on (relative) ranks; we consider two classes of such statistics in this section.  First, we introduce the so--called simple linear rank statistics, which are of the form
\begin{equation}\label{eq:simplelinearrank}
V_{ij}=\sum_{t=1}^n c_{nt} \, g(R_{ij}^{(t)}/(n+1))\,, \qquad 1\le i< j\le p\,,
\end{equation}
where $g$ is a Lipschitz function (also called score function), and $(c_{nt})$ with $c_{nt}=n^{-1} f(t/(n+1))$ for a Lipschitz function $f$ and $\sum_{t=1}^n c_{nt}^2 >0$ are called the regression constants.
An example of such a simple linear rank statistic is Spearman's $\rho$, which will be discussed in detail in Section~\ref{sec:spearman}. For $1\leq i<j\leq p$ the relative ranks $(R_{ij}^{(t)})_{t=1}^n$ depend on the vectors $\x_i$ and $\x_j$, where $\x_k=(X_{k1},\ldots, X_{kn})$ for $1\leq k\leq p$. We assume that the  vectors $\x_1,\ldots \x_p$ are independent. It is worth mentioning that the ranks $(Q_{it})$ remain the same if we transform the marginal distributions to the (say) standard uniform distribution. Thus, the joint distribution of $(R_{ij}^{(t)})_{t=1}^n$, and thereby the distribution of $V_{ij}$, does not depend on the distribution of $\x_i$ or $\x_j$. Therefore, we may assume without loss of generality that the random vectors $\x_1,\ldots, \x_p$ are identically distributed. We can write $V_{ij}=g_{n,V}(\x_i,\x_j)$ for a measurable function $g_{n,V}:\R^{2n}\to\R$.

Next, we consider rank-type $U$-statistics of order $m<n$ of the form
\begin{equation}\label{eq:ustatistic}
U_{ij}= \frac{1}{n(n-1)\cdots (n-m+1)}\sum_{1\le t_1\neq  \cdots  \neq t_m\le n} h( (X_{i t_1},X_{j t_1}), \ldots,  (X_{i t_m},X_{j t_m}))\,,
\end{equation}
where the symmetric kernel $h$ is such that $U_{ij}$ depends only on $(R_{ij}^{(t)})_{t=1}^n$. An important example of a rank-type $U$- statistic is Kendall's $\tau$, which will be studied in Section~\ref{sec:kendall}. For more examples we refer to \cite{han:chen:liu:2017} and references therein.
As for simple linear rank statistics, we are able to write $U_{ij}=g_{n,U}(\x_i,\x_j)$, where $g_{n,U}:\R^{2n}\to\R$ is a measurable function and $\x_1,\ldots\x_p$ are iid random vectors. 
\smallskip

An interesting property of rank-based statistics is the following pairwise independence. We also note that they are generally not
mutually independent.
\begin{lemma}[Lemma C4 in \cite{han:chen:liu:2017}] 
For $1\le i<j\le p$, let $\Psi_{ij}$ be a function of the relative ranks $\{R_{ij}^{(t)}, t=1,\ldots,n\}$. Assume $\x_1,\ldots,\x_p$ are independent. Then for any $(i,j) \neq (k,l)$, $i< j, k< l$, the random variables $\Psi_{ij}$ and $\Psi_{kl}$ are independent.
\end{lemma}
As an immediate consequence we obtain pairwise independence of $(U_{ij})$ and $(V_{ij})$, respectively.

\begin{lemma}\label{lem:UVind} 
For any $(i,j) \neq (k,l)$, $i< j, k< l$, the random variables $V_{ij}$ and $V_{kl}$ are independent and identically distributed. Moreover, $U_{ij}$ and $U_{kl}$ are independent  and identically distributed.
\end{lemma}

We now want to standardize $U_{ij}$ and $V_{ij}$. By independence of $(X_{it})$, we have
\begin{equation*}
\E[V_{ij}] = \ov g_n \sum_{t=1}^n c_{nt}\,, \quad 
\var(V_{ij})= \frac{1}{n-1} \sum_{t=1}^n ( g(t/(n+1))- \ov g_n)^2 \sum_{s=1}^n (c_{ns}- \ov c_n)^2\,,
\end{equation*}
where $\ov g_n = n^{-1} \sumt g(t/(n+1))$ is the sample mean of $g(Q_{11}/(n+1)),\ldots, g(Q_{1n}/(n+1))$ and $\ov c_n= \sumt c_{nt}$. Expectation and variance of $U_{ij}$ can also be calculated analytically.
We set
\begin{equation*}
\mu_V= \E[V_{12}]\,, \sigma_V^2= \var(V_{12}) \quad \text{ and } \quad \mu_U= \E[U_{12}]\,, \sigma_U^2=\var(U_{12})\,,
\end{equation*}
and define the standardized versions of $U_{ij}$ and $V_{ij}$ by
\begin{equation}\label{eq:UVtilde}
\wt V_{ij} = (V_{ij}-\mu_V)/\sigma_V \quad \text{ and } \quad \wt U_{ij}=(U_{ij}-\mu_U)/\sigma_U\,,\quad 1\le i<j\le p.
\end{equation}
It is well--known that $\wt V_{ij}$ and $\wt U_{ij}$ are asymptotically standard normal and the following lemma provides a complementary large deviation result.
\begin{lemma} \cite[p.404-405]{kallenberg:1982} \label{lem:tailUV} 
Suppose that the kernel function $h$ is bounded and non-degenerate. Then we have for $x=o(n^{1/6})$ that
\begin{equation*}
\P(\wt U_{12} >x) = \ov \Phi (x) (1+o(1)), \qquad \nto \,.
\end{equation*}
Assume that the score function $g$ is differentiable with bounded Lipschitz constant and that the constants $(c_{nt})_t$ satisfy 
\begin{equation}\label{eq:condcnt}
\max_{1\le t\le n} |c_{nt}- \ov c_n|^2 \le \frac{C^2}{n^{ 2/3}} \sumt (c_{nt}-\ov c_n)^2 \,, \quad \Big| \sumt (c_{nt}-\ov c_n)^3 \Big|^2 \le \frac{C^2}{n}
\Big| \sumt (c_{nt}-\ov c_n)^2 \Big|^3\,,
\end{equation}
where $C$ is some constant. Then it holds for $x=o(n^{1/6})$
\begin{equation*}
\P(\wt V_{12} >x) = \ov \Phi (x) (1+o(1)),\qquad \nto \,.
\end{equation*}
\end{lemma}
For a discussion of \eqref{eq:condcnt}, see \cite[p.405]{kallenberg:1982}. 
To proceed we need to find a suitable scaling and centering sequences  for $\wt V_{ij}$ and $\wt U_{ij}$, respectively, such that the conditions of Theorem~\ref{thm:maxg} are fulfilled.
For an iid standard normal sequence $(X_i)$ it is known that
\beao
\lim_{p \to \infty} \P\Big(\wt d_p \big(\max_{i=1,\ldots,p} X_i-\wt d_p\big)\le x\Big)=\exp(-\ex^{-x})=\Lambda(x)\,,
\qquad x\in\R\,, 
\eeao
where $\wt d_p=\sqrt{2\log p} - \tfrac{\log\log p+\log 4\pi}{2(2\log p)^{1/2}}$; see Embrechts et al. \cite[Example~3.3.29] {embrechts:kluppelberg:mikosch:1997}.
Since we are dealing with $p(p-1)/2$ random variables $(V_{ij})$ and $(U_{ij})$, respectively, which are asymptotically standard normal,  $ d_p =\wt d_{p(p-1)/2}$ seems like a reasonable choice for scaling and centering sequences.

Our main result for rank-statistics is the following.
\begin{theorem}\label{thm:U}
(a) Suppose that the kernel function $h$ is bounded and non-degenerate.
 If $p = \exp(o(n^{1/3}))$, the following \pp\ \con\ holds
\begin{equation}\label{eq:NU}
N_n^U := \sum_{1\le i<j\le p} \vep_{ d_p( \wt U_{ij}- d_p)} \cid N :=\sum_{i=1}^{\infty} \vep_{-\log \Gamma_i}\,,\qquad \nto\,,
\end{equation}
where $\Gamma_i= E_1+\cdots+E_i$, $i\ge 1$, and $(E_i)$ are iid standard exponential, i.e., $N$ is a Poisson random measure with mean measure $\mu(x,\infty)=\e^{-x}$,
$x\in\R$. \\
(b) Assume that the score function $g$ is differentiable with bounded Lipschitz constant and that the constants $(c_{nt})_t$ satisfy \eqref{eq:condcnt}.
Then if $p=\exp(o(n^{1/3}))$, it holds that
\begin{equation}\label{eq:NV}
N_n^V := \sum_{1\le i<j\le p} \vep_{ d_p( \wt V_{ij}- d_p)} \cid N\,,\qquad \nto\,.
\end{equation}
\end{theorem}
\begin{proof}
We start with the proof of \eqref{eq:NV} for which we will use Theorem \ref{thm:maxg}, as $\x_1,\ldots\x_p$ are iid and $g_{n,V}$ is a measurable function. Therefore, we only have to show that for $x\in \R$ it holds
\begin{enumerate}
\item[(1)] $\frac{p(p-1)}{2}\P(\widetilde{V}_{12}>x_p)\to \e^{-x}$ as $n\to\infty$,
\item[(2)] $p^3\P(\widetilde{V}_{12}>x_p, \widetilde{V}_{13}>x_p)\to 0$ as $n\to\infty$,
\end{enumerate}
where $x_p=x/d_p+d_p$. We will begin with the proof of (1). Since $x_p\sim d_p=o(n^{1/6})$ we get by Lemma \ref{lem:tailUV}  
\begin{align*}
\frac{p(p-1)}{2}\P(\widetilde{V}_{12}>x_p)=\frac{p(p-1)}{2}\bar{\Phi}(x_p)(1+o(1))
\end{align*}
and by Mill's ratio we have (writing $\tilde{p}=\tfrac{p(p-1)}{2}$)
\begin{align*}
\tilde{p}\,\bar{\Phi}(x_p)\sim \tilde{p}\frac{1}{\sqrt{2\pi}x_p}\e^{-x_p^2/2}\sim \tilde{p}\frac{1}{\sqrt{2\pi}\sqrt{2\log \tilde{p}}}\e^{-\log\tilde{p}+(\log\log\tilde{p})/2+(\log (4\pi))/2}\e^{-x}=\e^{-x}.
\end{align*}
Regarding (2), we note that, by Lemma~\ref{lem:UVind}, $\widetilde{V}_{12}$ and $\widetilde{V}_{13}$ are independent. Thus, we get 
\begin{align*}
p^3\P(\widetilde{V}_{12}>x_p, \widetilde{V}_{13}>x_p)=p^3\P(\widetilde{V}_{12}>x_p)^2=p^3(\bar{\Phi}(x_p)(1+o(1)))^2\to 0,\qquad n\to\infty,
\end{align*}
where we used Lemma \ref{lem:tailUV} and Mill's ratio in the last two steps. 
That completes the proof of \eqref{eq:NV}. The proof of \eqref{eq:NU} follows by analogous arguments.
\end{proof}
\begin{remark}\label{thm:hanmax}{\em
Theorem~\ref{thm:U} is a generalization of Theorems 1 and 2 in \cite{han:chen:liu:2017} who proved 
under the conditions of Theorem \ref{thm:U} and if $p = \exp(o(n^{1/3}))$ that
\begin{equation*}
\limn \P\Big( \max_{1\le i<j\le p} \wt V_{ij}^2 -4 \log p + \log \log p \le x\Big) = \exp\Big(-\tfrac{1}{\sqrt{8 \pi}} \e^{-x/2} \Big)\,, \quad x\in \R\,
\end{equation*}
and
\begin{equation*}
\limn \P\Big( \max_{1\le i<j\le p} \wt U_{ij}^2 -4 \log p + \log \log p \le x\Big) = \exp\Big(-\tfrac{1}{\sqrt{8 \pi}} \e^{-x/2} \Big)\,, \quad x\in \R\,.
\end{equation*}
}\end{remark}

As in Theorem \ref{thm:rectime}, we additionally conclude point process convergence for the record times of the maxima of $V_{ij}$ and $U_{ij}$. 
To this end, we investigate the sequence $(\max_{1\leq i<j\leq k}U_{ij})_{k\geq 1}$. This sequence jumps at time $k$ if one of the random variables $U_{1k},\ldots, U_{k-1,k}$ is larger than every $U_{ij}$ for $1\leq i<j\leq k-1$. Between these jump (or record) times the sequence is constant.

Let $L^U$ be this sequence of record times defined by
\begin{align*}
    L^U(1)&=1,\\
    L^U(k+1)&=\inf\{\ell>L^U(k):\max\limits_{1\le i< j\leq \ell} U_{ij}>\max\limits_{1\le i<j\le L^U(k)} U_{ij}\},\qquad k\in\N,
\end{align*}
and let $L^V$ be constructed analogously.
\begin{theorem}
Under the conditions of Theorem \ref{thm:U} it holds that
\begin{align*}
    \sum_{k=1}^p\varepsilon_{p^{-1}L^V(k)}\cid J\quad \text{and}\quad \sum_{k=1}^p\varepsilon_{p^{-1}L^U(k)}\cid J,
\end{align*}
in $\mathcal{M}(0,1]$, the space of point measures on $(0,1]$, where $J$ is a Poisson random measure with mean measure $\nu(a,b)=\log(b/a)$ for $0<a<b\leq 1$.
\end{theorem}

As in Corollary \ref{rem:recdis}, we can draw conclusions on the index of the last and second last jump before or at $p$.  Let $\zeta^U(p)$ be the number of records among $\max_{1\leq i<j\leq 2}U_{ij},\ldots, \max_{1\leq i<j\leq p}U_{ij}$. Then, as $n\to \infty$, we have for $x,y \in (0,1]$
\begin{enumerate}
    \item[(1)] $\P(p^{-1}L^U(\zeta^U(p))\leq x)\to \P(J(x,1]=0)=x$,\\
    
    \item[(2)] $\P(p^{-1}L^U(\zeta^U(p))\leq x, p^{-1}L^U(\zeta^U(p)-1)\leq y)\to y+y\log(x/y)$ for $x>y$,\\
    
    \item[(3)] $\P(p^{-1}(L^U(\zeta^U(p))-L^U(\zeta^U(p)-1))\leq x)\to x(1-\log(x))$,
\end{enumerate}
where (3) gives information about how much time elapses between the second last and the last jump of $(\max_{1\leq i<j\leq k}U_{ij})_{k\geq 1}$ before or at $p$.

\subsubsection{Kendall's tau}\label{sec:kendall}
Kendall's tau is an example of a rank-type U-statistic with bounded kernel. For $i\neq j$ Kendall's tau $\tau_{ij}$ measures the ordinal association between the two sequences $(X_{i1},\ldots,X_{in})$ and $(X_{j1},\ldots,X_{jn})$. It is defined by
\begin{align*}
\tau_{ij}&=\frac{2}{n(n-1)}\sum_{1\leq t_1< t_2\leq n}\operatorname{sign}(X_{it_1}-X_{it_2})\operatorname{sign}(X_{jt_1}-X_{jt_2})\\
&=\frac{2}{n(n-1)}\sum_{1\leq t_1< t_2\leq n}\operatorname{sign}(R_{ij}^{(t_2)}-R_{ij}^{(t_1)}),
\end{align*}
where the function $\operatorname{sign}:\R\to\{1,0,-1\} $ is given by $\operatorname{sign}(x)=x/|x|$ for $x\neq 0$ and $\operatorname{sign}(0)=0$. 
An interesting property of Kendall's tau is that there exists a representation as a sum of independent random variables. We could not find this representation in the literature. Therefore, we state it here. The proof can be found in Section~\ref{proof}.

\begin{proposition}\label{prop:tau}
We have
\begin{equation*}
\tau_{12} \eid \frac{4}{n(n-1)} \sum_{i=1}^{n-1} D_i\,,
\end{equation*}
where $(D_i)_{i\ge 1}$ are independent random variables with $D_i$ being uniformly distributed on the numbers $-i/2, -i/2+1 , \ldots, i/2$.
\end{proposition}
From Proposition~\ref{prop:tau} we deduce
$\E[\tau_{ij}]=0$ and $\Var(\tau_{ij})=\tfrac{2(2n+5)}{9n(n-1)}$.
The next result is a corollary of Theorem~\ref{thm:U}.
\begin{corollary}
Under the conditions of Theorem \ref{thm:U} we have
\begin{equation}\label{eq:Ntau}
N_n^{\tau} := \sum_{1\le i<j\le p} \vep_{ d_p( \tau_{ij}/\sqrt{\Var(\tau_{ij}})- d_p)} \cid N =\sum_{i=1}^{\infty} \vep_{-\log \Gamma_i}\,,\qquad \nto\,.
\end{equation}
\end{corollary}

\subsubsection{Spearman's rho}\label{sec:spearman}
An example of a simple linear rank statistic is Spearman's rho, which is a measure of rank correlation that assesses how well the relationship between two variables can be described using a monotonic function. Recall that $Q_{ik}$ and $Q_{jk}$ are the ranks of $X_{ik}$ and $X_{jk}$ among $\{X_{i1},\ldots , X_{in}\}$ and $\{X_{j1},\ldots, X_{jn}\}$, respectively, and write $q_n=(n+1)/2$ for the average rank. Then for $1\leq i\neq j\leq p$ Spearman's rho is defined by
\begin{align*}
\rho_{ij}&=\frac{\sum_{k=1}^n(Q_{ik}-q_n)(Q_{jk}-q_n)}{\big(\sum_{k=1}^n(Q_{ik}-q_n)^2\sum_{k=1}^n(Q_{jk}-q_n)^2\big)^{1/2}}\\
&=\frac{12}{n(n^2-1)}\sum_{k=1}^n\Big(k-\frac{n+1}{2}\Big)\Big(R_{ij}^{(k)}-\frac{n+1}{2}\Big).
\end{align*}
For mean and variance we get
\begin{align}\label{rhoEV}
\E[\rho_{ij}]=0 \quad\text{and}\quad\Var(\rho_{ij})=1/(n-1).
\end{align}
Therefore, we obtain the following corollary of Theorem~\ref{thm:U}.
\begin{corollary}\label{cor:spearman}
Under the conditions of Theorem \ref{thm:U} it holds that
\begin{align*}
N_n^\rho := \sum_{1\le i<j\le p} \vep_{ d_p(  \rho_{ij}/\sqrt{\Var(\rho_{ij})}- d_p)} \cid N\,.
\end{align*}
\end{corollary}
The next auxiliary result allows us to transfer the weak convergence of a sequence of point processes to a another sequence of point processes, provided that the maximum distance between their points tends to zero in probability.
\begin{proposition}\label{prop:ext}
For arrays $(X_{i,n})_{i,n\ge 1}$ and $(Y_{i,n})_{i,n\ge 1}$ of real-valued random variables, let $N^X_n=\sum_{i=1}^p \vep_{X_{i,n}}$ and assume that $N^X_n\cid N$. Consider a point process $N^Y_n=\sum_{i=1}^p \vep_{Y_{i,n}}$. If 
$$\max_{i=1,\ldots,p}|X_{i,n}-Y_{i,n}| \cip 0,$$
then $N^Y_n\cid N$.
\end{proposition}
\begin{example}{\em 
It turns out that there is an interesting connection between Spearman's rho and Kendall's tau. By \cite[p.318]{hoeffding1948aclass} we can write Spearman's rho as
\begin{align}\label{reprho}
\rho_{ij}=\frac{n-2}{n+1}r_{ij}+\frac{3\tau_{ij}}{n+1},\quad\quad1\leq i\neq j\leq p,
\end{align}
where 
\begin{align*}
r_{ij}=\frac{3}{n(n-1)(n-2)}\sum_{1\leq t_1\neq t_2\neq t_3\leq n}\operatorname{sign}(X_{it_1}-X_{it_2})\operatorname{sign}(X_{jt_1}-X_{jt_3})
\end{align*}
is the major part of Spearman's rho.
Therefore, $r_{ij}$ is a U-statistic of degree three with an asymmetric bounded kernel and with
\begin{align}\label{rEV}
\E[r_{ij}]=0\quad\text{and}\quad\Var(r_{ij})=\frac{n^2-3}{n(n-1)(n-2)},\quad\quad 1\leq i\neq j \leq p.
\end{align}
We now use Proposition \ref{prop:ext} and Corollary~\ref{cor:spearman} to show that
\begin{equation}\label{eq:rhddrsdg}
N_n^{r} := \sum_{1\le i<j\le p} \vep_{ d_p(  r_{ij}/\sqrt{\Var(r_{ij})}- d_p)} \cid N:= \sum_{i=1}^{\infty} \vep_{-\log \Gamma_i}\,,\qquad \nto\,.
\end{equation}
For this purpose we consider the following difference\begin{align*}
     d_p\Big(  \frac{\rho_{ij}}{\sqrt{\Var(\rho_{ij})}}
     - d_p\Big)- d_p\Big(  \frac{r_{ij}}{\sqrt{\Var(r_{ij})}}- d_p\Big)
     = d_p\Big(\frac{\rho_{ij}}{\sqrt{\Var(\rho_{ij})}}-\frac{r_{ij}}{\sqrt{\Var(r_{ij})}}\Big).
\end{align*}
By \eqref{rhoEV}, \eqref{rEV} and \eqref{reprho} this expression is asymptotically equal to
\begin{align*}
    \frac{d_p}{\sqrt{n}}(\rho_{ij}-r_{ij})=\frac{3d_p}{\sqrt{n}(n+1)}(\tau_{ij}-r_{ij}).
\end{align*}
Since $|\tau_{ij}|$ and $|r_{ij}|$ are bounded above by constants, we deduce that
\begin{align*}
\max_{1\leq i<j\leq p}\Big|\frac{3d_p}{\sqrt{n}(n+1)}(\tau_{ij}-r_{ij})\Big|\cip 0 \,, \qquad \nto,
\end{align*}
which verifies the condition in Proposition~\ref{prop:ext}. Since $N_n^{\rho}\cid N$ by Corollary~\ref{cor:spearman}, we conclude the desired \eqref{eq:rhddrsdg}.
}\end{example}

\subsection{Sample covariances}
An important field of current research is the estimation and testing of high-dimensional covariance structures. It finds application in genomics, social science and financial economics; see \cite{cai:2017} for a detailed review and more references. 
Under quite general assumptions Xiao et al.~\cite{xiao2013asymptotic} investigated the maximum off-diagonal entry of a high-dimensional sample covariance matrix. We impose the same model assumptions (compare \cite[p. 2901-2903]{xiao2013asymptotic}), but instead of the maximum we  study the point process of off-diagonal entries.

We start by describing the model and spelling out the required assumptions.
Let $\x_1,\ldots,\x_n$ be $p$-dimensional iid random vectors with $\x_i=(X_{1i},\ldots,X_{pi})$, where $\E[X_{ji}]=0$ for $1\leq j\leq p$ and $\bar{X}_j:=\frac{1}{n}\sum_{k=1}^n X_{jk}$. Denote $\Sigma=(\sigma_{i,j})_{1\leq i,j\leq p}$ as the covariance matrix  of the vector $\x_1$ and assume $\sigma_{i,i}=1$ for $1\leq i\leq p$. The empirical covariance matrix $(\hat{\sigma}_{i,j})_{1\leq i,j\leq p}$ is given by
\begin{align*}
\hat{\sigma}_{i,j}=\frac{1}{n}\sum_{k=1}^n (X_{ik}-\bar{X}_{i})(X_{jk}-\bar{X}_j),\quad\quad 1\leq i,j\leq p.
\end{align*}
A fundamental problem in high-dimensional inference is to derive the asymptotic distribution of $\max_{1\le i<j\le p} |\hat{\sigma}_{i,j}-\sigma_{i,j}|$.
Since the $\hat{\sigma}_{i,j}$'s might have different variances we need to standardize $\hat{\sigma}_{i,j}$ by
$\theta_{i,j}=\Var(X_{i1}X_{j1})$, which can be estimated by
\begin{align*}
\hat{\theta}_{i,j}=\frac{1}{n}\sum_{k=1}^n\Big[(X_{ik}-\bar{X}_{i})(X_{jk}-\bar{X}_j)-\hat{\sigma}_{i,j}\Big]^2.
\end{align*}
We are interested in the points
\begin{align*}
M_{i,j}:=\frac{|\hat{\sigma}_{i,j}-\sigma_{i,j}|}{\sqrt{\hat\theta_{i,j}}},\quad\quad 1\leq i<j\leq p.
\end{align*} 
Let $\mathcal{I}_n=\{(i,j): 1\leq i<j\leq p\}$ be an index set. We use the following notations to formulate the required conditions:
\begin{align*}
\mathcal{K}_n(t,r)&=\sup_{1\leq i\leq p}\E[\exp(t|X_{i1}|^r)],\\
\mathcal{M}_n(r)&=\sup_{1\leq i\leq p}\E[|X_{i1}|^r],\\
\theta_n&=\inf_{1\leq i<j\leq p}\theta_{i,j},\\
\gamma_n&=\sup_{\substack{\alpha,\beta\in \mathcal{I}_n\\ \alpha\neq \beta}} |\operatorname{Cor}(X_{i1}X_{j1},\,X_{k1}X_{l1})|,\quad \text{for}\, \alpha=(i,j),\,\beta=(k,l),\\
\gamma_n(b)&=\sup_{\alpha\in \mathcal{I}_n}\sup_{\substack{A\subset \mathcal{I}_n\\ |A|=b}}\inf_{\beta\in A}|\operatorname{Cor}(X_{i1}X_{j1},\,X_{k1}X_{l1})|\quad \text{for}\, \alpha=(i,j),\,\beta=(k,l).
\end{align*}
Now, we can draft the following conditions.
\begin{itemize}
\item[(B1) ] $\liminf\limits_{n \to \infty}\theta_n>0$. 
\item[(B2) ] $\limsup\limits_{n\to\infty}\gamma_n<1$.
\item[(B3) ] $\gamma_n(b_n)\log(b_n)=o(1)$ for any sequence $(b_n)$ such that $b_n\to\infty$.
\item[(B3')] $\gamma_n(b_n)=o(1)$ for any sequence $(b_n)$ such that $b_n\to\infty$ and for some $\varepsilon>0$
\begin{align*}
\sum_{\alpha,\beta\in \mathcal{I}_n} (\operatorname{Cov}(X_{i1}X_{j1},\,X_{k1}X_{l1}))^2=O(p^{4-\varepsilon})\quad \text{for}\, \alpha=(i,j),\,\beta=(k,l).
\end{align*}
\item[(B4) ] For some constants $t>0$ and $0<r\leq 2$, $\limsup\limits_{n\to\infty} \mathcal{K}_n(t,r)<\infty $, and
\begin{align*}
\log p=\begin{cases}o(n^{r/(4+r)}),\quad&\text{if}\,\,0<r<2, \\o(n^{1/3}(\log n)^{-2/3}),\quad&\text{if}\,\,r=2.\end{cases}
\end{align*}
\item[(B4') ] $\log p=o(n^{r/(4+3r)})$, $\limsup\limits_{n\to\infty} \mathcal{K}_n(t,r)<\infty $ for some constants $t>0$ and $r>0$.
\item[(B4'')] $p=O(n^q)$ and $\limsup\limits_{n\to\infty}\mathcal{M}_n(4q+4+\delta)<\infty$ for some constants $q>0$ and $\delta>0$.
\end{itemize}
 To be able to adopt parts of the proof of Theorem 2 in \cite{xiao2013asymptotic} we consider (instead of $(M_{i,j})$) the transformed points $(W_{i,j})$ given by
\begin{align*}
W_{i,j}:=\frac{1}{2}(n\,M_{i,j}^2-4\log p+\log\log p+\log 8\pi),\quad\quad 1\leq i<j\leq p,
\end{align*}
and we define the point processes
\begin{align*}
N_n^{(W)}:=\sum_{1\leq i<j\leq p} \vep_{W_{i,j}}\,.
\end{align*}
\begin{theorem}\label{thm:cov}
Let $\E[\x_1]=0$ and $\sigma_{i,i}=1$ for all $i$, and assume \textup{(B1)} and \textup{(B2)}. Then under any one of the following conditions:
\begin{itemize}
\item[(i)] \textup{(B3)} and \textup{(B4)},
\item[(ii)] \textup{(B3')} and \textup{(B4')},
\item[(iii)] \textup{(B3)} and \textup{(B4'')},
\item[(iv)] \textup{(B3')} and \textup{(B4'')},
\end{itemize}
it holds, that
\begin{align*}
N_n^{(W)}\cid N =\sum_{i=1}^{\infty} \vep_{-\log \Gamma_i}\,,\qquad \nto\,,
\end{align*}
where $\Gamma_i= E_1+\cdots+E_i$, $i\ge 1$, and $(E_i)$ are iid standard exponential, i.e., $N$ is a Poisson random measure with mean measure $\mu(x,\infty)=\e^{-x}$,
$x\in\R$.
\end{theorem}
\begin{proof}
Under condition (i) set $\mathcal{E}_n=n^{-(2-r)/(4(r+4))}$ if $0<r<2$, and $\mathcal{E}_n=n^{-1/6}(\log n)^{1/3}(\log p)^{1/2}$ if $r=2$. Under condition (ii) let $\mathcal{E}_n=(\log p)^{1/2}n^{-r/(6r+8)}$. Under (i) or (ii) we set
\begin{align*}
\tilde{X}_{ik}=X_{ik}\mathds{1}_{\{|X_{ik}|\leq T_n\}}-\E[X_{ik}\mathds{1}_{\{|X_{ik}|\leq T_n\}}],\quad\quad 1\leq i\leq p,\,\,\, 1\leq k\leq n,
\end{align*}
where $T_n=\mathcal{E}_n(n/(\log p)^3)^{1/4}$.
Under conditions (iii) and (iv) we set
\begin{align*}
\tilde{X}_{ik}=X_{ik}\mathds{1}_{\{|X_{ik}|\leq  n^{1/4}/\log n\}},\quad\quad 1\leq i\leq p,\,\,\, 1\leq k\leq n.
\end{align*}
Additionally, we define
$\tilde{\sigma}_{i,j}=\E[\tilde{X}_{i1}\tilde{X}_{j1}]$ and $\tilde{\theta}_{i,j}=\Var[\tilde{X}_{i1}\tilde{X}_{j1}]$.
We consider 
\begin{align*}
M_{1;i,j}=\frac{1}{\sqrt{\tilde{\theta}_{i,j}}}\Big|\frac{1}{n}\sum_{k=1}^n\tilde{X}_{ik}\tilde{X}_{jk}-\tilde{\sigma}_{ij}\Big|
\end{align*}
and the transformed points 
\begin{align*}
W_{1;i,j}=\frac{1}{2}(n\,M_{1;i,j}^2-4\log p+\log\log p+\log 8\pi).
\end{align*}
We will show that $N_n^{(W_1)}:=\sum_{1\leq i<j\leq p} \vep_{W_{1;i,j}}\cid N$ and thus by Proposition \ref{prop:ext} $N_n^{(W)}\cid N$. 

Therefore, we first apply Kallenberg's Theorem as in the proof of Theorem \ref{thm:maxg}. We set 
\begin{align*}
B=\bigcup_{k=1}^q B_k\subset \R
\end{align*}
with disjoint intervals $B_k=(r_k,s_k]$ and show
 \begin{enumerate}
 \item[(1)]
$\lim\limits_{n\to\infty} \mu_n^{(W_1)}(B)=\mu(B)$,
\item[(2)]
$\lim\limits_{n\to\infty} \P(N_n^{(W_1)}(B)=0)=\e^{-\mu(B)}$,
 \end{enumerate}
where $\mu_n^{(W_1)}(B)=\E[N_n^{(W_1)}(B)]$ and $\mu$ is defined by $\mu(B_k)=\e^{-r_k}-\e^{-s_k}$.
 
 From the proof of Theorem 2 of \cite[p. 2910, 2913-2914]{xiao2013asymptotic} we know that the conditions of \cite[Lemma 6]{xiao2013asymptotic} are satisfied. Furthermore, from the proof of Lemma 6 \cite[p. 2909-2910]{xiao2013asymptotic} we get that for $z\in \R$ and
\begin{align*}
z_n=(4\log p-\log \log p -\log 8\pi+2z)^{1/2}
\end{align*}
and $d\in \N$
\begin{align*}
\lim_{n\to\infty}\sum_{\substack{A\subset \mathcal{I}_n\\|A|=d}}\P(\sqrt{n}M_{1;i_1,j_1}>z_n,\ldots,\sqrt{n}M_{1;i_d,j_d}>z_n)=\frac{\e^{-dz}}{d!},
\end{align*}
which is equivalent to 
\begin{align}\label{fin}
\lim_{n\to\infty}\sum_{\substack{A\subset \mathcal{I}_n\\|A|=d}}\P(W_{1;i_1,j_1}>z,\ldots,W_{1;i_d,j_d}>z)=\frac{\e^{-dz}}{d!},
\end{align}
where $A=\{(i_1,j_1),\ldots,(i_d,j_d)\}$. Therefore, we get for $d=1$ 
\begin{align*}
\lim_{n\to\infty}\mu_n^{(W_1)}(B)=\lim_{n\to\infty}\sum_{k=1}^q \sum_{(i,j)\in \mathcal{I}_n}\P(W_{1;i,j}\in B_k)=\sum_{k=1}^q (\e^{-r_k}-\e^{-s_k}) =\mu(B).
\end{align*}
which proves (1).
Regarding (2), we use that $1-\P(N_n^{(W_1)}(B)=0)=\P\Big(\bigcup_{1\leq i<j\leq p} A_{i,j}\Big)$, where $A_{i,j}=\{W_{1;i,j}\in B\}$. By Bonferroni's inequality we have for every $k\geq 1$,
\begin{align}\label{eq:drhhsds}
&\sum_{d=1}^{2k}(-1)^{d-1}\sum_{\substack{A\subset \mathcal{I}_n\\|A|=d}}P_{A,B}
\leq\P\Big(\bigcup_{1\leq i<j\leq p} A_{i,j}\Big)
\leq\sum_{d=1}^{2k-1}(-1)^{d-1}\sum_{\substack{A\subset \mathcal{I}_n\\|A|=d}}P_{A,B},
\end{align}
where $A=\{(i_1,j_1),\ldots,(i_d,j_d)\}$ and $P_{A,B}=\P(W_{1;i_1,j_1}\in B,\ldots,W_{1;i_d,j_d}\in B)$.
 First letting $\nto$ and then $k \to \infty$, we deduce from \eqref{fin} and \eqref{eq:drhhsds} that
\begin{align*}
\lim\limits_{n\to\infty} \P(N_n^{(W_1)}(B)=0)=1-\sum_{d=1}^\infty(-1)^{d-1}\frac{(\mu(B))^d}{d!}=\sum_{d=0}^\infty(-1)^{d}\frac{(\mu(B))^d}{d!}=\e^{-\mu(B)}.
\end{align*}
This proves (2)
and we get $N_n^{(W_1)}\cid N$. By Proposition \ref{prop:ext} it remains to show
\begin{align*}
\max_{1\leq i<j\leq p}|W_{1;i,j}-W_{i,j}|=\frac{n}{2} \max_{1\leq i<j\leq p}|M^2_{1;i,j}-M^2_{i,j}|\cip 0.
\end{align*}
Fortunately, this is shown in the course of the proof of Theorem 2 of \cite[p. 2911-2916]{xiao2013asymptotic}.
\end{proof}

The following examples are motivated by \cite[p. 2903-2905]{xiao2013asymptotic}.
\begin{example}[Physical dependence]\em{
Assume that $\x_1=(X_{11},\ldots,X_{p1})$ is distributed as a stationary process of the following form. For a measurable function $g$ and a sequence of iid random variables $(\epsilon_i)_{i\in\Z}$  we set $\x_1=(X_{11},\ldots,X_{p1})$ with
\begin{align*}
X_{i1}= g(\epsilon_i,\epsilon_{i-1},\ldots),\quad\quad i\ge 1,
\end{align*}
and let $\x_k$, $2\leq k\leq n$, be iid copies of $\x_1$.
Moreover, for an iid copy $(\epsilon'_i)_{i\in\Z}$ of $(\epsilon_i)_{i\in\Z}$ and
\begin{align*}
X'_{i1}=g(\epsilon_i,\ldots,\epsilon_1,\epsilon'_0,\epsilon_{-1},\ldots)
\end{align*}
we define the {\it physical dependence measure} of order $q$ (see \cite{wu:2005})
\begin{align*}
\delta_q(i)=\E\big[ |X_{i1}-X'_{i1}|^q\big]^{1/q} \quad \text{ and } \quad 
\Psi_q(k)=\Big[\sum_{i=k}^\infty (\delta_q(i))^2\Big]^{1/2}.
\end{align*}}
Then, we conclude from Lemma 3 of \cite{xiao2013asymptotic} and Theorem \ref{thm:cov} the following statement.

Assume that $0<\Psi_4(0)<\infty$ and $\Var(X_{i1}X_{j1})>0$ for all $i,j\in\Z$ and \\$|\operatorname{Cor}(X_{i1}X_{j1},X_{k1}X_{l1})|<1$ for all $i,j,k,l$, such that they are not all the same. Then, if either one of the conditions
\begin{enumerate}
\item[(i)] $\Psi_q(k)=o(1/\log k)$ as $k\to\infty$ and one of the assumptions (B4) and (B4') or
\item[(ii)] $\sum_{j=0}^p(\Psi_4(j))^2=O(p^{1-\delta})$ for some $\delta>0$ and one of the assumptions (B4') or (B4'') 
\end{enumerate}
is satisfied, we have 
\begin{align*}
N_n^{(W)}\cid N,\qquad n\to\infty.
\end{align*}

As a special case we consider the linear process  $X_{i1}=\sum_{j=0}^\infty a_j\epsilon_{i-j}$, where the $\epsilon_j$ are iid with $\E[\epsilon_j]=0$ and $\E[|\epsilon_j|^q]<\infty$ with $q\ge 4$ and for $a_j\in\R$ it holds that $\sum_{j=0}^\infty a_j^2 \in(0,\infty)$. Then the physical dependence measure is given by $\delta_q(j)=|a_j|\,\E\big[ |\epsilon_0-\epsilon'_0|^q\big]^{1/q}$. Moreover, the conditions $0<\Psi_4(0)<\infty$ and $\Var(X_{i1}X_{j1})>0$ for all $i,j\in\Z$ and $|\operatorname{Cor}(X_{i1}X_{j1},X_{k1}X_{l1})|<1$ for all $i,j,k,l$, such that they are not all the same, are fulfilled. If $a_j=j^{-\beta}\ell(j)$, where $1/2<\beta<1$ and $\ell$ is a slowly varying function, then $(X_{i1})$ is a long memory process. The smaller the value of $\beta$, the stronger is the dependence between the $(X_{i1})$. If one of the assumptions (B4) or (B4') is satisfied, then condition (i) is fulfilled for every $\beta\in (1/2,1)$.

\end{example}

\begin{example}[Non-stationary linear processes]\em{
As in the previous example, $\x_1, \ldots, \x_n$ are iid random vectors. Now $\x_1=(X_{11},\ldots,X_{p1})$ is given by
\begin{align*}
X_{i1}=\sum_{t\in\Z}f_{i,t}\epsilon_{i-t}, \qquad i\ge 1,
\end{align*}
where $(\epsilon_i)_{i\in\Z}$ is a sequence of iid random variables with mean zero, variance one and finite fourth moment and the sequences $(f_{i,t})_{t\in\Z}$ satisfy $\sum_{t\in \Z}f_{i,t}^2=1$. Let $\kappa_4$ be the fourth cumulant of $\epsilon_0$ and 
\begin{align*}
h_n(k)=\sup_{1\leq i\leq p}\Big(\sum_{|t|=\lfloor k/2\rfloor}^\infty f_{i,t}^2\Big)^{1/2}.
\end{align*}
Assume that $\kappa_4>-2$ and 
\begin{align}\label{assump}
\limsup\limits_{n\to\infty}\sup_{1\leq i<j\leq p}\Big|\sum_{t\in\Z}f_{i,i-t}f_{j,j-t}\Big|<1.
\end{align}
By Section 3.2 of \cite[p. 2904-2905]{xiao2013asymptotic} and Theorem \ref{thm:cov} we get the following result. If either
\begin{enumerate}
\item[(i)] $h_n(k_n)\log k_n=o(1)$ for any positive sequence $k_n$ such that $k_n\to\infty$ as $n\to\infty$ and one of the assumptions (B4) and (B4') or
\item[(ii)] $\sum_{k=1}^p(h_n(k))^2=O(p^{1-\delta})$ for some $\delta>0$ and one of the assumptions (B4') or (B4'') 
\end{enumerate}
holds, then we have $N_n^{(W)}\cid N$ as $ n\to\infty$.

To illustrate these assumptions we consider the special case $\x_1:=(\epsilon_1,\ldots,\epsilon_p) A_n$, where  $A_n\in \R^{p\times p}$ is a deterministic, symmetric matrix with $(A_n)_{i,j}= a_{ij}$ for $1\leq i,j\leq p$. We assume that $\sum_{t=1}^p a_{it}^2=1$ for every $1\leq i\leq p$.\\
The covariance matrix of $\x_1$ is given by $\operatorname{Cov}(\x_1)=A_nA_n^T$ with $(A_nA_n^T)_{ij}=\sum_{t=1}^p a_{it}a_{jt}$. Observe that the diagonal entries are equal to $1$. To satisfy assumption \eqref{assump} we have to assume that the entries except for the diagonal are asymptotically smaller than $1$, i.e.
\begin{align*}
\limsup\limits_{n\to\infty}\sup_{1\leq i<j\leq p}\Big|\sum_{t=1}^p a_{it}a_{jt}\Big|<1.
\end{align*}
We set 
\begin{align*}
h_n(k)=\sup_{1\leq i\leq p}\Big(\sum_{t=1}^{i-\lfloor k/2 \rfloor} a_{it}^2+\sum_{t=\lfloor k/2 \rfloor+i}^{p}a_{it}^2\Big)^{1/2}
\end{align*}
as a measure of how close the matrices $A_n$ are to diagonal matrices. 
For the point process convergence either (i) or (ii) has to be satisfied for $h_n$.
}\end{example}

\section{Proofs of the Results}\label{proof}
\subsection{Proofs of the results in Section \ref{sec1}}

\begin{proof}[Proof of Theorem \ref{thm:maxg}]
We will follow the lines of the proof of Theorem~2.1 in \cite{dabrowski2002poisson}.
Since the mean measure $\eta$ has a density, the limit process $M$ is simple and we can apply Kallenberg's Theorem (see for instance \cite[p.233, Theorem 5.2.2]{embrechts:kluppelberg:mikosch:1997} or \cite[p.35, Theorem 4.7]{kallenberg:1983}). Therefore, it suffices to prove that for any finite union of bounded rectangles
\begin{align*}
R=\bigcup_{k=1}^q A_k\times B_k\subset S,\qquad\text{with}\qquad A_k=\otime{l=1}{m} (r_k^{(l)},s_k^{(l)}],\quad B_k=(r_k^{(m+1)},s_k^{(m+1)}],
\end{align*}
 it holds that 
 \begin{enumerate}
 \item
$\lim\limits_{n\to\infty} \eta_n(R)=\eta(R)$,
\item
$\lim\limits_{n\to\infty} \P(M_n(R)=0)=\e^{-\eta(R)}$.
 \end{enumerate}
Without loss of generality we can assume that the $A_k$'s are chosen to be disjoint.  First we will show (1). Set $T:=T_{(1,2,\ldots, m)}=g_n(\x_1,\x_2,\ldots, \x_m)$. If $q=1$ we get
\begin{align*}
\eta_n(R)=\E[M_n(A_1\times B_1)]&=\sum_{\iv:\iv/p \in A_1}\P(T_{\iv} \in B_1)\\
&\sim p^{m}\, \prod_{l=1}^m (s_1^{(l)}-r_1^{(l)})\P(T\in B_1)\,.
\end{align*}
Since assumption (A1) implies $p^m /(m!)\,\P(T\in B_1)\to \mu(B_1)$,  we obtain the convergence $\eta_n(R)\to \eta(R)$ as $n\to \infty$. The case $q\geq 1$ follows by
\begin{align*}
\eta_n(R)=\sum_{k=1}^q\eta_n(A_k\times B_k)\to\sum_{k=1}^q\eta(A_k\times B_k)=\eta(R),\qquad n\to\infty.
\end{align*}

To show (2), 
we let $P_n$ be the probability mass function of the Poisson distribution with mean $\eta_n(R)$. Then we have
\begin{align*}
|\P(M_n(R)=0)-\P(M(R)=0)|&\leq |\P(M_n(R)=0)-P_n(0)|+|P_n(0)-\P(M(R)=0)|\\
&=|\P(M_n(R)=0)-P_n(0)|+o(1),
\end{align*}
where the last equality holds by (1). Therefore, we only have to estimate $|\P(M_n(R)=0)-P_n(0)|$. For this we employ the Stein-Chen method (see \cite{barbour:1992} for a discussion). The Stein equation for the Poisson distribution $P_n$ with mean $\eta_n(R)$ is given by
\begin{align}\label{steq}
\eta_n(R) x(j+1)-jx(j)=\1_{\{j=0\}}-P_n(0),\qquad j\geq 0. 
\end{align}
This equation is solved by the function
\begin{align*}
x(0)&=0\\
x(j+1)&=\frac{j!}{\eta_n(R)^{j+1}}\e^{\eta_n(R)}(P_n(\{0\})-P_n(\{0\})P_n(\{0,\ldots,j\})\,,\quad j=0,1,\ldots 
\end{align*}
By \eqref{steq} we see that
\begin{align}\label{stein}
|\P(M_n(R)=0)-P_n(0)|=|\E[\eta_n(R) x(M_n(R)+1)-M_n(R)x(M_n(R))]|.
\end{align}
Therefore, we only have to estimate the right hand side of \eqref{stein} and to this end we set
\begin{align*}
D&:=\{\kv :\,\kv=(k_1, k_2,\ldots,k_m), 1\leq k_1<k_2<\ldots<k_m\leq p\},\\
I_\kv&:=\sum_{i=1}^q\1_{A_i}(\kv/p)\1_{B_i}(T_\kv),\\
\eta_\kv&:=\E[I_\kv].
\end{align*}
For $\kv\in D$ let
\begin{align*}
D_{1\kv}&:=\{\lv\in D:\, \lv_i\neq k_j,\, i,j=1,2,\ldots,m\}\quad\text{and}\\
D_{2\kv}&:=\{\lv\in D:\, \lv\neq \kv,\, \lv_i=k_j\,\,\text{for some}\,\, i,j=1,2,\ldots,m\}.
\end{align*}
Then we have the disjoint union $D=D_{1\kv}\overset{.}{\cup} D_{2\kv}\overset{.}{\cup} \{\kv\}$, and therefore,
\begin{align*}
M_n(R)=\sum_{\lv\in D}I_\lv=\sum_{\lv\in D_{1\kv}}I_\lv+\Big(I_\kv+\sum_{\lv\in D_{2\kv}}I_\lv\Big)=:M^{(1)}_n(\kv)+M^{(2)}_n(\kv).
\end{align*}
Now, we bound \eqref{stein} by
\begin{align}
&\Big|\sum_{\kv\in D}\E[\eta_\kv x(M_n(R)+1)-I_\kv x(M_n(R))]\Big|\nonumber\\
&\leq \Big|\sum_{\kv\in D}\eta_\kv\E[x(M_n(R)+1)-x(M_n^{(1)}(\kv)+1)]\Big|+\Big|\sum_{\kv\in D}[\E[I_\kv x(M_n(R))]-\eta_\kv\E[x(M_n^{(1)}(\kv)+1)]]\Big|. \label{sum}
\end{align}
It suffices to show that both terms in \eqref{sum} tend to zero as $n\to\infty$. 
From \cite[p.400]{barbour1984poisson} we have the following bound for the increments of the solution of Stein's equation
\begin{align}
\Delta x:=\sup_{j\in\N_0}|x(j+1)-x(j)|\leq\min(1,1/\eta_n(R)).\label{zuwachs}
\end{align}

Using \eqref{zuwachs} the first term of \eqref{sum} is bounded above by 
\begin{align*}
 \sum_{\kv\in D}\eta_\kv|\E[x(M_n^{(1)}(\kv)+M_n^{(2)}(\kv)+1)-x(M_n^{(1)}(\kv)+1)]|\leq \sum_{\kv\in D}\eta_\kv\E[M_n^{(2)}(\kv)].
\end{align*}
Using the definitions of $\eta_\kv$ and $M_n^{(2)}(\kv)$, we get

\begin{align}
&\sum_{\kv\in D}\eta_\kv\E[M_n^{(2)}(\kv)]\nonumber\\
&=\sum_{\kv\in D}\Big(\sum_{i=1}^q\1_{A_i}(\tfrac{\kv}{p})\P(T\in B_i)\Big)\Big(\sum_{\lv\in D}\sum_{j=1}^q\1_{A_j}\big(\tfrac{\lv}{p}\big)\P(T\in B_j)-\sum_{\lv\in D_{1\kv}}\sum_{j=1}^q\1_{A_j}(\tfrac{\lv}{p})\P(T\in B_j)\Big)\nonumber\\
&=\sum_{i=1}^q\sum_{j=1}^q\P(T\in B_i)\P(T\in B_j)\sum_{\kv\in D}\1_{A_i}(\tfrac{\kv}{p})\sum_{\lv\in D}\1_{A_j}(\tfrac{\lv}{p})\nonumber\\
&-\sum_{i=1}^q\sum_{j=1}^q\P(T\in B_i)\P(T\in B_j)\sum_{\kv\in D}\1_{A_i}(\tfrac{\kv}{p})\sum_{\lv\in D_{1\kv}}\1_{A_j}(\tfrac{\lv}{p})\,. \label{lan}
\end{align}
Since by assumption (A1) it holds that $p^m\P(T\in B_i)\to m!(\mu(r_i^{(m+1)})-\mu(s_i^{(m+1)}))$ as $n\to\infty$, and
\begin{align*}
\frac{1}{p^m}\sum_{\kv\in D}\1_{A_i}(\tfrac{\kv}{p})&\to\lambda_m(A_i),\qquad n\to\infty\\
\frac{1}{p^{2m}}\sum_{\kv\in D}\sum_{\lv\in D_{1\kv}}\1_{A_i}(\tfrac{\kv}{p})\1_{A_j}(\tfrac{\lv}{p})&\to\lambda_m(A_i)\lambda_m(A_j)\qquad n\to\infty,
\end{align*}
\eqref{lan} and thus the first term of \eqref{sum} tend to zero as $n\to\infty$.  As every $I_\kv$ only depends on $T_\kv$ and because $D_{1\lv}$ only contains elements which have no component in common with $\lv$, $M_n^{(1)}(\lv)$ and $I_\lv$ are independent. Therefore, the second term of \eqref{sum} equals
\begin{align}
\Big|\sum_{\kv\in D}\E[I_\kv(x(M_n(R))-x(M_n^{(1)}(\kv)+1)]\Big|\leq \Delta x\sum_{\kv\in D}\E[I_\kv(M_n^{(2)}(\kv)-1)],\label{del}
\end{align}
where the last inequality follows from \eqref{zuwachs}. Since $I_\kv\leq 1$  because the $A_i$ are disjoint, the \rhs~in \eqref{del} is bounded above by
\begin{align}
&\Delta x\sum_{\kv\in D}\E[I_\kv\sum_{\lv\in D_{2\kv}}I_\lv]\nonumber\\
&=\Delta x \sum_{\kv\in D}\sum_{\lv\in D_{2\kv}}\E\Big[\Big(\sum_{i=1}^q\1_{A_i}(\tfrac{\kv}{p})\1_{B_i}(T_\kv) \Big)\Big(\sum_{j=1}^q\1_{A_j}(\tfrac{\lv}{p})\1_{B_j}(T_\lv)\Big)\Big]\nonumber\\
&=\Delta x \sum_{i=1}^q\sum_{j=1}^q\sum_{\kv\in D}\sum_{\lv\in D_{2\kv}}\1_{A_i}(\tfrac{\kv}{p})\1_{A_j}(\tfrac{\lv}{p})\P(T_\kv\in B_i,T_\lv\in B_j)\label{D2kr}
\end{align}
We set $D_{2\kv,r}:=\{\lv\in D:|\{\lv_1,\ldots, \lv_m,k_1,\ldots,k_m\}|=2m-r\}$.  Notice that $\dot{\bigcup}_{r=1}^{m-1} D_{2\kv,r}=D_{2\kv}$. Therefore, \eqref{D2kr} is equal to
\begin{align*}
\Delta x \sum_{i=1}^q\sum_{j=1}^q\sum_{r=1}^{m-1}\sum_{\kv\in D}\sum_{\lv\in D_{2\kv,r}}\1_{A_i}(\tfrac{\kv}{p})\1_{A_j}(\tfrac{\lv}{p})\P(T_\kv\in B_i,T_\lv\in B_j).
\end{align*}
By assumption (A2), we have $p^{2m-r}\P(T_\kv\in B_i,T_\lv\in B_j)\to 0$ for $r=1, \ldots m-1$ as $n\to\infty$. Additionally, it holds that
\begin{align*}
\frac{1}{p^{2m-r}}\sum_{\kv\in D}\sum_{\lv\in D_{2\kv,r}}\1_{A_i}(\tfrac{\kv}{p})\1_{A_j}(\tfrac{\lv}{p})=O(1),\qquad n\to\infty.
\end{align*}
Consequently the second term of \eqref{sum} tends to zero as $n\to\infty$.
This completes the proof.
\end{proof}

\begin{proof}[Proof of Proposition \ref{prop:extr}] 
We proceed similarly to the proof of Proposition 4.20 of \cite{resnick:1987}.
We want to show that $Y_n\cid Y$. Therefore, we define a map from the space of point measures $\mathcal{M}(S)$ to $D(0,1]$, the space of right continuous functions on $(0,1]$ with finite limits existing from the left,  and show that this map is continuous. Then, the Proposition follows by the continuous mapping theorem. 

To this end, for a point measure $\mathbf{m}=\sum_{k=1}^\infty\varepsilon_{(t_k,y_k)} \in \mathcal{M}(S)$ we define $V_1:\mathcal{M}(S)\to D(0,1]$ through
\begin{align*}
    V_1(\mathbf{m})=V_1\Big(\sum_{k=1}^\infty\varepsilon_{(t_k,y_k)}\Big)=\begin{cases} &v_\mathbf{m}:(0,1]\to (v,w)\\
 &v_\mathbf{m}(t)=\begin{cases} \bigwedge\limits_{k:t_k\leq t} y_k,\,\, \mathbf{m}(((0,1]^{m-1}\times (0,t]\times (v,w))\cap S)>0\\
 \bigwedge\limits_{k:t_k=t^*}y_k,\,\, \text{otherwise},
\end{cases}
\end{cases}
\end{align*}
where $t^*=\sup\{s>0:\mathbf{m}(((0,1]^{m-1}\times (0,s]\times (v,w))\cap S)=0\}$. $V_1$ is well-defined except at $\mathbf{m}\equiv 0$. Recalling the definition of $N_n$ in \eqref{eq:defNnsdfsddv}, we note that $V_1(N_n)(t)=Y_n(t)$ and $V_1(N)(t)=Y(t)$ for $0<t\leq 1$. \\

We will start by proving the continuity of $V_1$ in the case, where $\mu(x)=-\log(H(x))$ and $H$ is the Gumbel distribution. In this case, $N$ has a.s.~the following properties
\begin{align}
    N(((0,1]^{m-1}\times\{1\}\times (-\infty,\infty))\cap S)&=0,\\
    N(((0,1]^{m-1}\times (0,t] \times (x,\infty))\cap S)&<\infty,\label{prop2}\\
    N(((0,1]^{m-1}\times [s,t] \times (-\infty,x))\cap S)&=\infty,
\end{align}
 for any $0<s<t<1$ and $x\in\R$. Therefore, we only have to show continuity at $\mathbf{m}\in\mathcal{M}(S)$ with these properties. Let $(\mathbf{m}_n)_n$ be a sequence of point measures in $\mathcal{M}(S)$, which converges vaguely to $\mathbf{m}$ ($\mathbf{m}_n\stackrel{v}{\rightarrow}\mathbf{m}$) as $n\to\infty$ (see \cite[p. 140]{resnick:1987}). Since $V_1(\mathbf{m})$ is right continuous there exists a right continuous extension on $[0,1]$, which we denote with $\widetilde{V_1(\mathbf{m})}$. Now choose $\beta<\widetilde{V_1(\mathbf{m})}(0)$ such that $\mathbf{m}(S_1\times \{\beta\})=0$. As $\mathbf{m}_n\stackrel{v}{\rightarrow}\mathbf{m}$, we can conclude from \cite[Proposition 3.12]{resnick:1987} that there exists a $1\leq q<\infty$ such that for $n$ large enough 
\begin{align*}
    \mathbf{m}_n(S_1\times (\beta,\infty))=\mathbf{m}(S_1\times (\beta,\infty))=q.
\end{align*}
We enumerate and designate the $q$ points in the following way $((t_i^{(n)},j_i^{(n)}),\, 1\leq i\leq q)$ with $0<t_{1,m}^{(n)}<\ldots<t_{q,m}^{(n)}<1$, where $t_{i,m}^{(n)}$ is the $m$-th component of $t_i^{(n)}$, such that by \cite[Proposition 3.13]{resnick:1987}
\begin{align*}
    \lim\limits_{n\to\infty}((t_i^{(n)},j_i^{(n)}),\, 1\leq i\leq q)=((t_i,j_i),\, 1\leq i\leq q),
\end{align*}
where $((t_i,j_i),\, 1\leq i\leq q)$ is the analogous enumeration of points of $\mathbf{m}$ in $S_1\times (\beta,\infty)$. Now choose 
\begin{align*}
\delta< \frac{1}{2} \min\limits_{1\leq i,j\leq q}\Vert t_i-t_j\Vert_2
\end{align*}
small enough so that the $\delta$-spheres of the distinct points of the set $\{(t_i,j_i)\}$ are disjoint and in $S_1\times [\beta,\infty)$. 
Pick $n$ so large that every $\delta$-sphere contains only one point of $\mathbf{m}_n$. Then set $\lambda_n:[0,1]\to [0,1]$ with $\lambda_n(0)=0$, $\lambda_n(1)=1$, $\lambda_n(t_{i,m})=t_{i,m}^{(n)}$ and $\lambda_n$ is linearly interpolated elsewhere on $[0,1]$. For this $\lambda_n$ it holds that
\begin{align*}
    \sup\limits_{0\leq t\leq 1}|\widetilde{V_1(\mathbf{m}_n)}(t)-\widetilde{V_1(\mathbf{m})}(\lambda_n(t))|&<\delta\quad \text{and} \quad
    \sup\limits_{0\leq t\leq 1}|\lambda_n(t)-t|<\delta.
\end{align*}
Thereby, we get
\begin{align*}
    \tilde{\mathcal{D}}(\widetilde{V_1(\mathbf{m}_n)},\widetilde{V_1(\mathbf{m})})=\mathcal{D}(V_1(\mathbf{m}_n),V_1(\mathbf{m}))<\delta,
\end{align*}
which finishes the proof. The Fr\'{e}chet and the Weibull case follow by similar arguments.
\end{proof}

\begin{proof}[Proof of Theorem \ref{thm:rectime}]
We will proceed similarly as in \cite[p.~217-218]{resnick:1987} using the continuous mapping theorem again. Since $Y$ is the restriction to $(0,1]$ of an extremal process (see \cite[Section 4.3]{resnick:1987}), it is a nondecreasing function, which is constant between isolated jumps. Let $D^\uparrow(0,1]$ be the subset of $D(0,1]$ that contains all functions with this property. Set 
\begin{align*}
V_2:&D^\uparrow(0,1]\to \mathcal{M}(0,1]\\
&x\mapsto \sum_{i=1}^\infty \varepsilon_{t_i},
\end{align*}
where $\{t_i\}$ are the discontinuity points of $x$. Then $V_2(Y_n)=\sum_{k=1}^p\varepsilon_{p^{-1}L(k)}$ and $V_2(Y)=\sum_{k=1}^\infty\varepsilon_{\tau_k}$, where $(\tau_k)_k$ is the sequence of discontinuity points of the extremal process generated by the Gumbel distribution $H=\Lambda$, c.f. above Theorem~\ref{thm:rectime}. By \cite[Theorem 5.4.7]{embrechts:kluppelberg:mikosch:1997} the point process $\sum_{k=1}^\infty\varepsilon_{\tau_k}$ is a PRM with mean measure $\nu(a,b)=\log(b/a)$ for $0<a<b\leq 1$. According to Proposition \ref{prop:extr}, it suffices to show that $V_2$ is continuous. Let $(x_n)_n$ be a sequence of functions in $D^\uparrow (0,1]$ with $\mathcal{D}(x_n,x)\to 0$ as $n\to\infty$ for an $x\in D^\uparrow (0,1]$. Then there exist $\lambda_n\in\Lambda_{[0,1]}$ such that
\begin{align}
    \sup\limits_{0\leq t\leq 1}|\tilde{x}_n(\lambda_n(t))-\tilde{x}(t)|&\to 0\quad \text{and}\label{sko1}\\
    \sup\limits_{0\leq t\leq 1}|\lambda_n(t)-t|&\to 0,\qquad \nto \,,\label{sko2}
\end{align}
where $\tilde{x}_n$ and $\tilde{x}$ are the right continuous extensions of $x_n$ and $x$ on $[0,1]$. We want to prove the vague convergence
\begin{align*}
V_2(x_n)=\sum_{i=1}^\infty \varepsilon_{t_i^{(n)}}\stackrel{v}{\rightarrow} V_2(x)=\sum_{i=1}^\infty \varepsilon_{t_i},
\end{align*}
where $\{t_i^{(n)}\}$ and $\{t_i\}$ are the discontinuity points of $x_n$ and $x$, respectively. Consider an arbitrary continuous function $f$ on $(0,1]$ with compact support contained in an interval $[a,b]$ with $0<a<b\leq 1$, and $x$ is continuous at $a$ and $b$. 
It suffices to show that
\begin{align}\label{contin}
    \lim_{n\to\infty}\sum_{i=1}^\infty f(t_i^{(n)})\1_{[a,b]}(t_i^{(n)})=\sum_{i=1}^\infty f(t_i)\1_{[a,b]}(t_i).
\end{align}
The functions $x_n,x\in D^\uparrow(0,1]$ have only finitely many discontinuity points in $[a,b]$. Therefore, only a finite number of terms in the sums are not equal to zero. Because of \eqref{sko1} and \eqref{sko2} the jump times on $[a,b]$ of $x_n$ are close to those of $x$, which proves \eqref{contin}. Hence, $V_2$ is continuous, which finishes the proof.
\end{proof}

\subsection{Proof of Proposition~\ref{prop:tau}}

Let $\pi$ denote the permutation of $\{1,\ldots,n\}$ induced by the order statistics of $X_{2 1}, \ldots, X_{2 n}$, i.e.,
\begin{equation*}
X_{2 \pi(1)} > X_{2 \pi(2)} > \cdots > X_{2 \pi(n)} \quad \as\,,
\end{equation*}
where the continuity of the distribution of $X$ was used to avoid ties.
We can rewrite $\tau_{12}$ as 
\begin{align}
\tau_{12}&= \tfrac{2}{n(n-1)}  \sum_{1\le s<t\le n} \sign(X_{1s}-X_{1t}) \sign(X_{2s}-X_{2t})\nonumber\\
&= \tfrac{2}{n(n-1)}  \sum_{1\le s<t\le n} \sign(X_{1\pi(s)}-X_{1\pi(t)}) \underbrace{\sign(X_{2\pi(s)}-X_{2\pi(t)})}_{=1 \quad \as}\nonumber\\
&\eid \tfrac{2}{n(n-1)} \sum_{1\le s<t\le n} \sign(X_{1s}-X_{1t}).\label{eq:esfe}
\end{align}
Let $\mathbf{q}_n=(q_1,\ldots,q_n)$ be a permutation of the set $\{1, \ldots,n\}$. If $i<j$ and $q_i>q_j$, we call the pair $(q_i, q_j)$ an inversion of the permutation $\mathbf{q}_n$. 

 Since the $X_{11},\ldots, X_{1n}$ are iid, the permutation
\begin{equation*}
\mathbf{q}_n= (Q_{11}, Q_{12}, \ldots,Q_{1n})
\end{equation*}
consisting of the ranks is uniformly distributed on the set of the $n!$ permutations of $\{1,\ldots,n\}$. By $I_n$ we denote the number of inversions of $\mathbf{q}_n$. For $s<t$, we have
\begin{equation*}
\sign(X_{1s}-X_{1t}) = \left\{
\begin{array}{cl}
1, & \quad \text{ if } Q_{1s}<Q_{1t}\,, \\
-1, & \quad \text{ if } Q_{1s}>Q_{1t} \, \Leftrightarrow \text{ inversion at } (s,t)\,. 
\end{array}
\right. 
\end{equation*}
In view of \eqref{eq:esfe}, this implies 
\begin{equation*}
\begin{split}
\tau_{12}
\eid \binom{n}{2}^{-1} \sum_{1\le s<t\le n} \sign(X_{1s}-X_{1t})
= \binom{n}{2}^{-1} \Big[  \binom{n}{2} - 2 \, I_n \Big] = 1-\tfrac{4}{n(n-1)} I_n\,.
\end{split}
\end{equation*}
By \cite[p.~479]{kendall:stuart:1973} or \cite[p.~3]{margolius:2001} (see also \cite{sachkov:1997}) the moment generating function of $I_n$ is
\begin{equation}\label{eq:mgf}
\E\Big[ \e^{t I_n} \Big] = \prod_{j=1}^n \frac{1-\e^{jt}}{j(1-\e^t)}\,, \quad t\in \R\,.
\end{equation}
We recognize that $\frac{1-\e^{jt}}{j(1-\e^t)}$ is the moment generating function of a uniform distribution on the integers $0, 1, \ldots, j-1$. Let $(U_i)_{i\ge 1}$ be a sequence of independent random variables such that $U_i$ is uniformly distributed on the integers $0, 1, \ldots, i$. We get
\begin{equation*}
\begin{split}
1-\tfrac{4}{n(n-1)} I_n & \eid 1-\tfrac{4}{n(n-1)} \sum_{i=1}^{n-1} U_i \eid \tfrac{4}{n(n-1)} \sum_{i=1}^{n-1} (U_i-i/2)\,,
\end{split}
\end{equation*}
establishing the desired result.

\subsection{Proof of Proposition~\ref{prop:ext}}
Our idea is to transfer the convergence of $N_n^X$ onto $N_n^Y$. To this end, it suffices to show (see \cite[Theorem 4.2]{kallenberg:1983})
that for any continuous function $f$ on $\R$ with compact support,
\begin{equation*}
\int f \dint N_n^Y - \int f \dint N_n^X \cip 0\,, \quad \nto\,.
\end{equation*}
Suppose the compact support of $f$ is contained in $[K+\gamma_0, \infty)$ for some $\gamma_0>0$ and $K\in \R$. Since $f$ is uniformly continuous, $\omega(\gamma):= \sup \{|f(x)-f(y)|: x,y \in \R, |x-y| \le \gamma\}$ tends to zero as $\gamma\to 0$. 
We have to show that for any $\vep >0$,
\begin{equation}\label{eq:dgse}
\lim_{\nto} \P \Big( \Big| \sum_{i=1}^p \Big( f(Y_{i,n})-f(X_{i,n}) \Big)\Big| >\vep \Big) =0\,.
\end{equation}
On the sets
\begin{equation}\label{eq:An}
A_{n,\gamma}= \Big\{ \max_{i=1,\ldots,p} \big| Y_{i,n}-X_{i,n} \big|  \le \gamma \Big\}\, ,\quad \gamma \in (0, \gamma_0) \,,
\end{equation}
we have 
\begin{equation*}
\big|f(Y_{i,n})-f(X_{i,n}) \big| \le \omega(\gamma) \,\1_{\{ X_{i,n} >K\}}\,.
\end{equation*}
Therefore, we see that, for $\gamma \in (0, \gamma_0)$,
\begin{equation}\label{eq:angamma}
\begin{split}
\P &\Big( \Big| \sum_{i=1}^p \Big( f(Y_{i,n})-f(X_{i,n}) \Big)\Big| >\vep, A_{n,\gamma} \Big)\\
&\le \P \Big( \omega(\gamma)\, \#\{1\le i \le p : X_{i,n} >K \} > \vep \Big)\\
&\le \frac{\omega(\gamma)}{\vep} \E\big[ \#\{1\le i \le p : X_{i,n} >K \}\big]\\
&= \frac{\omega(\gamma)}{\vep} \E N_n^X((K,\infty])\\
&\to \frac{\omega(\gamma)}{\vep} \E N((K,\infty])\,, \qquad \nto\,.
\end{split}
\end{equation}
By assumption, it holds $\lim_{\nto} \P(A_{n,\gamma}^c) =0$. Thus, letting $\gamma\to 0$ establishes \eqref{eq:dgse}.

\bibliography{libraryjohannes}
\end{document}